\documentclass[a4paper]{amsart}
\usepackage{amssymb, enumerate, array, mathrsfs, bbm, verbatim}
\usepackage{diagrams}
\pdfoutput=1
\title{The Algebra of Filters of a Cubic Algebra}
\date{\dateheader}
\author{Colin G. Bailey}
\address{School of Mathematics,  Statistics \& Operations Research\\
Victoria University of Wellington\\
Wellington\\
New Zealand}
\email{Colin.Bailey@vuw.ac.nz}
\date{\dateheader}
\author{Joseph S. Oliveira}
\address{
Pacific Northwest National Laboratories\\
Richland, Washington\\
U.S.A.}
\email{Joseph.Oliveira@pnl.gov}
\date{\today}
\let\rsf\mathscr

\def\caret{\mathbin{\hat{\hphantom{m}}}} 
\def\zero{{\mathbf 0}}
\def\one{{\mathbf 1}}

\def\leftGen{{[\kern-1.1pt[}}
\def\rightGen{{]\kern-1.1pt]}}

\providecommand{\meet}{\mathbin{\wedge}}
\providecommand{\join}{\mathbin{\vee}}
\newcommand{\comp}[1]{\overline{#1}}
\ifx\upharpoonright\undefined
     \def\restrict{\hbox{\rm\kern0.166em\accent"12\kern-0.536em$\vert$\kern0.3em}}%
  \else
     \def\restrict{\upharpoonright}%
  \fi
\makeatletter
\def\twoSet#1#2{\left\{%
\vphantom{#2}#1\thinspace\right|\nolinebreak[3]\left.%
  #2%
  \vphantom{#1}%
  \right\}%
}
\def\oneSet#1{\left\lbrace#1\right\rbrace}

\newif\if@nstr
\def\setstrfalse{\let\if@nstr=\iffalse}
\def\setstrtrue{\let\if@nstr=\iftrue}
\def\@nstr #1#2{
\def\@@nstr ##1#1##2##3\@@nstr{\ifx
\@nstr ##2\setstrfalse \else \setstrtrue \fi }
\@@nstr #2#1\@nstr \@@nstr}
\def\@separate#1|#2@{\setFront{#1}\setBack{#2}}
\def\lb#1\rb{\@nstr|{#1} \if@nstr \@separate#1 @ \twoSet{\@setFront}{\@setBack}%
\else \@separate |{#1 }@ \oneSet{\@setBack}\fi%
}
\def\setFront#1{\def\@setFront{#1}}
\def\setBack#1{\def\@setBack{#1}}
\def\Set#1{\lb{#1}\rb}

\def\oneBrk#1{\left\langle#1\right\rangle}
\def\twoBrk#1#2{\left\langle%
\vphantom{#2}#1\thinspace\right|\nolinebreak[3]\left.%
  #2%
  \vphantom{#1}%
  \right\rangle%
}
\def\brk<#1>{\@nstr|{#1} \if@nstr \@separate#1 @ \twoBrk{\@setFront}{\@setBack}%
\else \@separate |{#1 }@ \oneBrk{\@setBack}\fi%
}

\def\thmref#1{\normalfont{theorem}~\ref{#1}}
\def\lemref#1{\normalfont{lemma}~\ref{#1}}

\def\corref#1{\normalfont{corollary}~\ref{#1}}

\theoremstyle{plain}
\newtheorem{thm}{Theorem}[section]
\newtheorem{lem}[thm]{Lemma}
\newtheorem{cor}[thm]{Corollary}
\newtheorem{prop}[thm]{Proposition}
\newtheorem{defn}[thm]{Definition}

\theoremstyle{remark}
{}
{}
{}
{}

\@ifpackageloaded{bbm}{

}{

}
\makeatother

\begin{document}
	\begin{abstract}
		In this paper we discuss the inclusion ordering on the filters of a 
		filter algebra,  a special type of Metropolis-Rota 
		algeba. Using embeddings into interval 
		algebras we show that the notion of ``untwisted'' gives rise 
		to a congruence relation on the group of g-filters. 
		We also show that there is a natural reflection operator on the class 
		of filters with an easily definable enveloping cubic subalgebra. 
	\end{abstract}
\maketitle
\section{Introduction}
One of the many variants on the Stone representation theorem for 
Boolean algebras shows that the poset of filters for a Boolean 
algebra is isomorphic to the poset of closed sets of its Stone space. 
This is another way to get a representation of Boolean algebras. 

In this paper we consider a similar idea as applied to cubic and 
MR-algebras,  in particular the subclass of filter algebras. 
We look at the class of all filters on a filter algebra 
$\mathcal L$ which is naturally ordered by inclusion. 

On this class we can define the relation ``relatively untwisted'' 
which simply says two filters sit the same way in a representation 
of $\mathcal L$ as a subalgebra of an interval algebra. 
We will show that this relation is a congruence 
relation on the group of g-filters of $\mathcal L$. 

As an extension of the ideas used in that proof we will show that 
there is a binary operation $\Delta$ definable on the class of all 
filters. Under reverse inclusion this class is almost cubic -- the 
failure is because we are actually working in a Heyting algebra. 

By defining two notions of Boolean elements we can construct an 
interval algebra from this class into which $\mathcal L$ 
embeds as a full subalgebra. The strongest notion of Boolean-ness 
gives rise to a cubic algebra in which the vertices are exactly the 
g-filters of $\mathcal L$ and the inner automorphisms are exactly the 
filter automorphisms of $\mathcal L$. 

Before we begin in earnest we recall some of the basic definitions and 
results of cubic algebra. 

\begin{defn}
    A \emph{cubic algebra} is a join semi-lattice with one and a binary 
    operation $\Delta$ satisfying the following axioms:
    \begin{enumerate}[a.]
        \item  if $x\le y$ then $\Delta(y, x)\join x = y$;
        
        \item  if $x\le y\le z$ then $\Delta(z, \Delta(y, x))=\Delta(\Delta(z, 
        y), \Delta(z, x))$;
        
        \item  if $x\le y$ then $\Delta(y, \Delta(y, x))=x$;
        
        \item  if $x\le y\le z$ then $\Delta(z, x)\le \Delta(z, y)$;
        
        \item[] Let $xy=\Delta(1, \Delta(x\join y, y))\join y$ for any $x$, $y$ 
        in $\mathcal L$. Then:
        
        \item  $(xy)y=x\join y$;
        
        \item  $x(yz)=y(xz)$;
    \end{enumerate}
\end{defn} 

\begin{defn}
    An \emph{MR-algebra} is a cubic algebra satisfying the MR-axiom:\\
    if $a, b<x$ then 
    \begin{gather*}
        \Delta(x, a)\join b<x\text{ iff }a\meet b\text{ does not exist.}
    \end{gather*}
\end{defn}

%
%
One way to think about cubic algebras is as  a family of connected 
implication algebras ``joined'' by the symmmetry group generated by 
$\Delta$. The relation $\preccurlyeq$ is one way we use to describe 
the way different implication algebras in this collection are related 
to one another. In some sense it is the spread out version of the 
partial order.  

\begin{defn}
    Let $\mathcal L$ be a cubic algebra and $a, b\in\mathcal L$. Then
    \begin{align*}
        a\preceq b &\text{ iff }\Delta(a\join b, a)\le b\\
        a\simeq b &\text{ iff }\Delta(a\join b, a)=b.
    \end{align*}
\end{defn}

\begin{lem}
    Let $\mathcal L$, $a$, $b$ be as in the definition. Then
    $$
    a\preceq b\text{ iff }b=(b\join a)\meet(b\join\Delta(1, a)).
    $$
\end{lem}
\begin{proof}
    See \cite{BO:eq} lemmas 2.7 and 2.12.
\end{proof}

\begin{prop}\label{prop:triv}
    Let $\mathcal L$ be a cubic algebra, and $p, q$ in $\mathcal L$ 
    are such that
    $p\preccurlyeq q$ and $p\meet q$ exists. Then $p\le q$.
\end{prop}
\begin{proof}
    We have $\Delta(p\join q, p)\le q$ as $p\preccurlyeq q$. Let
    $a=p\meet q$. Then
    \begin{align*}
        a&\le q\\
        \Delta(p\join q, a)&\le\Delta(p\join q, p)\le q\\
        \text{Hence }\qquad
        p\join q&=a\join\Delta(p\join q, a)\\
        &\le q
        \end{align*}
        and so $p\le q$.
\end{proof}

\begin{cor}\label{cor:triv}
    Let $\mathcal L$ be a cubic algebra, and $p, q$ in $\mathcal L$ 
    are such that
    $p\simeq q$ and $p\meet q$ exists. Then $p= q$.    
\end{cor}

\begin{lem}\label{lem:MRprec}
	If $\mathcal L$ is a cubic algebra then 
    $\mathcal L$ is an MR-algebra iff for all $x, y$ there is some $z$ 
    such that $z\preccurlyeq x$ and $z\preccurlyeq y$.
\end{lem}
\begin{proof}
	For the left to right implication use $z=x\caret y$. For the other 
	direction if $x, y\in\mathcal L$ then let $z\preccurlyeq x$ and 
	$z\preccurlyeq y$. Then $x, y\in\mathcal L_{z}$ and so $x\caret y$ 
	exists -- as it exists in $\mathcal L_{x}$. 
\end{proof}

There are several kinds of embeddings that are of interest in studying 
cubic algebras. In this paper we will use the notion of full embedding:
\begin{defn}\label{def:fullEmbedd}
	Let $f\colon\mathcal L\to\mathcal M$ be a cubic embedding. Then
	$f$ is \emph{full} iff for all $m<\one$ in $\mathcal M$ there is a 
	$l<\one$ in $\mathcal L$ with $m\le f(l)$. 
\end{defn}

\subsection{Implication Algebras}
Cubic algebras abound. The simplest construction is using an 
implication algebra. 

Let $\mathcal I$ be an implication algebra. We define
$$
\rsf I(\mathcal I)=\Set{\brk<a, b> | a, b\in\mathcal I, a\join 
b=1\text{ and }a\meet b\text{ exists}}
$$
ordered by 
$$
\brk<a, b>\le\brk<c, d>\text{ iff }a\le c\text{ and }b\le d. 
$$
This is a partial order that is an upper semi-lattice with join 
defined by
$$
\brk<a, b>\join\brk<c, d>=\brk<a\join c, b\join d>
$$
and a maximum element $\one=\brk<1, 1>$. 

We can also define a $\Delta$ function by
$$
\text{if }\brk<c, d>\le\brk<a, b>\text{ then }
\Delta(\brk<a, b>, \brk<c, d>)=\brk<a\meet(b\to d), b\meet(a\to c)>. 
$$

There is a natural embedding of $\mathcal I$ into $\rsf I(\mathcal 
I)$ given by
$$
e_{\mathcal I}(a)=\brk<1, a>. 
$$

For later reference we note that
$$
\Delta(\one, \brk<a, b>)=\brk<1\meet(1\to b), 1\meet(1\to a)>=\brk<b, a>
$$
for all $\brk<a, b>$. 

Generally an isomorphism from one cubic algebra $\mathcal L$ to 
another of the form $\rsf I(\mathcal I)$ for some implication algebra 
$\mathcal I$ is not canonical. For this reason we 
have the notion of presentation. 

\begin{defn}\label{def:present}
	Let $\mathcal L$ be a cubic algebra. A \emph{presentation of 
	$\mathcal L$} is an pair $\brk<\mathcal I, \varphi>$ where 
	$\varphi\colon\mathcal L\to\rsf I(\mathcal I)$ is a cubic isomorphism. 
\end{defn}

\subsection{Filters and Filter Algebras}
\subsubsection{Filter Algebras}
In the special case that $\rsf F$ is an implication lattice 
there is a canonical Boolean algebra $B_{\rsf F}$ in which $\rsf 
F$ sits as an ultrafilter -- defined by 
\begin{align*}
	B_{\rsf F}&=\rsf F\times\Set{0, 1}\\
	\brk<a, i>\join\brk<b, j>&=
	\begin{cases}
		\brk<a\join b, 1> & \text{ if }i=j=1  \\
		\brk<a\meet b, 0> & \text{ if }i=j=0  \\
		 \brk<b\to a, 1> & \text{ if }i=1,\ j=0  \\
		 \brk<a\to b, 1> & \text{ if }i=0, \ j=1
	\end{cases}\\
	\brk<a, i>\meet\brk<b, j>&=
	\begin{cases}
		\brk<a\meet b, 1> & \text{ if }i=j=1  \\
		\brk<a\join b, 0> & \text{ if }i=j=0  \\
		 \brk<a\to b, 0> & \text{ if }i=1,\ j=0  \\
		 \brk<b\to a, 0> & \text{ if }i=0, \ j=1
	\end{cases}\\
	\comp{\brk<a, i>}&=\brk<a, 1-i>\\
	1&=\brk<1, 1>\\
	0&=\brk<1, 0>. 
\end{align*}

The mapping $\iota_{\rsf F}\colon f\mapsto\brk<f, 1>$ embeds $\rsf F$ as an ultrafilter of
$B_{\rsf F}$. This mapping is also an implication embedding and so 
there is a natural embedding 
from $\rsf I(\rsf F)$ into $\rsf I(B_{\rsf F})$ given by
$$
\brk<a, b>\mapsto\brk<{\brk<a, 1>}, {\brk<b, 1>}>. 
$$
Note that there is a commutative diagram:
\begin{diagram}
	\rsf F & \rTo^{\iota_{\rsf F}} &  B_{\rsf F} \\
	 \dTo^{e_{\rsf F}} &  & \dTo_{e_{B_{\rsf F} }}  \\
	\rsf I(\rsf F) & \rTo_{\rsf I(\iota_{\rsf F})} & \rsf I(B_{\rsf F})
\end{diagram}

\begin{defn}\label{def:filAlg}
	A cubic algebra $\mathcal L$ is a \emph{filter algebra} iff
	$\mathcal L$ is isomorphic to $\rsf I(\rsf F)$ for some implication 
	lattice $\rsf F$. 
\end{defn}

\subsubsection{Filters in Cubic Algebras}
Most of the filters we will look at arise in cubic algebras. 

On these filters there are several interesting 
constructions that lead to cubic operations. First intersection. 
\begin{lem}\label{lem:interTwoFil}
	Let $\mathcal L$ be a cubic algebra and $\rsf F$ and $\rsf G$ be two 
	filters. Then
	$$
	\rsf F\cap\rsf G=\Set{f\join g | f\in\rsf F\text{ and }g\in\rsf G}. 
	$$
\end{lem}
\begin{proof}
	The RHS set is clearly a subset of both $\rsf F$ and $\rsf G$. 
	
	And if $z\in\rsf F\cap\rsf G$ then $z=z\join z$ is in the RHS set. 
\end{proof}

\begin{defn}\label{def:cup}
	Let $\rsf F, \rsf G$ be two $\mathcal L$-filters. Then
	$\rsf F\vee\rsf G$ is defined iff $\rsf F\cup\rsf G$ has fip, in 
	which case it is the filter generated by $\rsf F\cup\rsf G$. 	
\end{defn}

\begin{lem}\label{lem:cup}
	If $\rsf F\vee\rsf G$ exists then it is equal to
	$\Set{f\meet g | f\in\rsf F\text{ and }g\in\rsf G}$. 
\end{lem}
\begin{proof}
	Let $S$ be this set. It is clearly contained in $\rsf F\join\rsf G$. 
	
	If $h\in\rsf F\join\rsf G$ then there is some $f\in\rsf F$ and $g\in\rsf 
	G$ such that $f\meet g\le h$. Hence
	\begin{align*}
		h & =h\join(f\meet g)  \\
		 & =(h\join f)\meet(h\join g)  \\
		 & \in S. 
	\end{align*}
\end{proof}

It is easy to show that these operations are commutative, associative,  
idempotent
and satisfy absorption. 
Distributivity also holds in a weak way. 
\begin{lem}\label{lem:distrib}
	Let $\rsf G, \rsf H, \rsf K$ be subfilters of a filter $\rsf F$. Then
	$$
	\rsf G\cap(\rsf H\join\rsf K)=(\rsf G\join\rsf H)\cap(\rsf G\join\rsf 
	K). 
	$$
\end{lem}
\begin{proof}
	Let $x=g\join(h\meet k)\in\rsf G\cap(\rsf H\join\rsf K) $. 
	Then $x=(g\join h)\meet(g\join k)$ is in 
	$(\rsf G\join\rsf H)\cap(\rsf G\join\rsf K)$. 
	
	Conversely if $x=(g_{1}\join h)\meet(g_{2}\join k)$ is in 
	$(\rsf G\join\rsf H)\cap(\rsf G\join\rsf K)$ then 
	$x\geq g_{1}\meet g_{2}\in\rsf G$ and $x\geq h\meet k\in \rsf 
	H\join\rsf K$ so that $x\in\rsf G\cap(\rsf H\join\rsf K)$.
\end{proof}

We can also define a relative complement which we defer until section 
\ref{sect:seven}. 

Filters generate subalgebras of cubic algebras that are always filter 
algebras. This gives a way a defining when two filters are 
``similar'' -- they generate the same subalgebra.

\begin{defn}\label{def:GenFilter}
	Let $\rsf F$ be a filter in a cubic algebra $\mathcal L$. Then
	\begin{enumerate}[(a)]
		\item $\displaystyle\leftGen\rsf F\rightGen$ is the subalgebra of $\mathcal L$ 
		generated by $\rsf F$. 
	\vspace{1.2mm}
		\item $\displaystyle\widehat{\rsf F}=\Set{\Delta(x, y) | y\le x, \ x, y\in\rsf F}$. 
	\end{enumerate}
\end{defn}

\begin{thm}\label{thm:GenFil}
	Let $\rsf F$ be a filter in a cubic algebra $\mathcal L$. Then
	$$
	\leftGen\rsf F\rightGen=\widehat{\rsf F}. 
	$$
\end{thm}
\begin{proof}
	See \cite{BO:cubFil} theorem 4.16. 
\end{proof}

\begin{defn}\label{def:gFilter}
	Let $\rsf F$ be a filter in a cubic algebra $\mathcal L$. Then
	$\rsf F$ is a \emph{generating filter} or a \emph{g-filter} iff
	$\leftGen\rsf F\rightGen=\mathcal L$. 
\end{defn}

\begin{defn}\label{def:similar}
	Let $\rsf F$ and $\rsf G$ be two filters in a cubic algebra $\mathcal L$.
	Then \emph{$\rsf F$ is similar to $\rsf G$} -- written $\rsf F\simeq 
	\rsf G$ -- iff $\leftGen\rsf F\rightGen=\leftGen\rsf G\rightGen$. 
\end{defn}

The fact that $\rsf F$ is a g-filter for $\leftGen\rsf F\rightGen$ 
will often be used in the following. The most important fact about 
g-filters is that they are naturally isomorphic as implication lattices. 

If $\rsf F$ is a g-filter for $\mathcal L$ then for all $x\in\mathcal 
L$ there are unique elements $\alpha_{\rsf F}(x)\geq \beta_{\rsf F}(x)$ 
of $\rsf F$ such that
$x=\Delta(\alpha_{\rsf F}(x),  \beta_{\rsf F}(x))$ see \cite{BO:cubFil} 
lemma 4.23 and theorem 4.29 which also shows that if 
$\rsf G$ is another generating filter then
$\alpha_{\rsf F}\restrict\rsf G$ is an implication homomorphism
and $\beta_{\rsf F}\restrict\rsf G$ is an implication isomorphism
from $\rsf G$ to $\rsf F$. 

Furthermore we have 
\begin{thm}\label{thm:isoFilAlg}
	Let $\rsf F$ be a g-filter for $\mathcal L$. Then the mapping 
	$$
	x\mapsto\brk<\Delta(\one, x)\join\beta_{\rsf F}(x), x\join\beta_{\rsf F}(x)>
	$$
	is an isomorphism from
	$\mathcal L$ to $\rsf I(\rsf F)$. 
\end{thm}

This mapping will be called the \emph{$\rsf F$-presentation of 
$\mathcal L$}. 

\begin{cor}\label{cor:gFilEqFAlg}
	$\mathcal L$ is a filter algebra iff $\mathcal L$ has a g-filter. 
\end{cor}

Using this theorem we can easily define an extension of $\beta_{\rsf 
G}\colon\rsf F\to\rsf G$ to a cubic automorphism $\varphi_{\brk<\rsf 
F, \rsf G>}$ of $\mathcal L$ as
the composite 
\begin{diagram}
	\mathcal L=\leftGen\rsf F\rightGen & \rTo^{\sim} & \rsf I(\rsf F) & \rTo^{\rsf 
	I(\beta_{\rsf G})} & \rsf I(\rsf G) & \rTo^{\sim} & \leftGen\rsf 
	G\rightGen=\mathcal L
\end{diagram}

\begin{lem}\label{lem:interTwoFilAA}
	Let $\mathcal L$ be a cubic algebra and $\rsf F$ and $\rsf G$ be two 
	filters such that $\leftGen\rsf F\rightGen =\leftGen\rsf G\rightGen$. Then
	$$
	\rsf F\cap\rsf G=\Set{f\join \beta_{\rsf G}(f) | f\in\rsf F}. 
	$$
\end{lem}
\begin{proof}
	The RHS set is clearly a subset of both $\rsf F$ and $\rsf G$. 
	
	And if $z\in\rsf F\cap\rsf G$ then $z=\beta_{\rsf G}(z)$ is in the RHS set. 
\end{proof}

\section{Twisted Filters}
Suppose that $\rsf F$ and $\rsf G$ are two g-filters for $\mathcal L$. 
Then we have an isomorphism from $\mathcal L$ to $\rsf I(\rsf F)$. 
This 
representation of $\mathcal L$ has a ``direction'' of decreasing 
dimension given by the embedding of $\rsf F$ into $\rsf I(\rsf F)$. 
This is perhaps more clearly seen by embedding further into $\rsf 
I(B_{\rsf F})$ where the vertices (in some sense) correspond to 
choosing a basis. It is interesting to see how $\rsf G$ is mapped 
across.

\begin{defn}\label{def:twistedFilter}\index{filter!twisted}
	Let $\mathcal L=\rsf I(B)$ be an interval algebra and $\rsf F\subseteq\mathcal L$
	be a filter. Then $\rsf F$  is \emph{twisted} iff 	
	there is no $b\in B$ such that $\rsf F\subseteq[[b, b], \mathbf 1]$. 
\end{defn}

\begin{defn}\label{def:unTwisting}
	Let $B_{1}, B_{2}$ be two Boolean algebras and $\rsf F$ a
	filter in $\rsf I(B_{1})$.  Let $e\colon\rsf I(B_{1})\to\rsf I(B_{2})$
	be a cubic embedding.  Then $\rsf F$ is \emph{untwisted along $e$} iff
	$e[\rsf F]$ is not twisted. 
\end{defn}

\begin{defn}\label{def:relTwisting}
	Let $\mathcal L$ be a cubic algebra and $\rsf F, \rsf G\subseteq\mathcal L$
	be two filters with $\rsf F\subseteq\leftGen\rsf G\rightGen$. 
	Then $\rsf F$  is \emph{twisted relative to $\rsf G$} iff 
	$\rsf F$ is twisted under the natural embedding 
	$\leftGen\rsf G\rightGen\simeq\rsf I(\rsf G)\to\rsf 
	I(B_{\rsf G})$. 
\end{defn}

The rest of this section provides a characterization of those filters 
that are untwisted relative to some fixed g-filter $\rsf F$. This is 
then used to show that this relation is an equivalence relation on 
the class of g-filters.  

\begin{thm}\label{thm:unTwist}
	Let $\rsf F$ and $\rsf G$ be two g-filters for $\mathcal L$. Then
	$\rsf F$ is untwisted relative to $\rsf G$ iff one of
	$\rsf F\cap\rsf G$ and $\Delta(\one, \rsf F)\cap\rsf G$ is
	principal. 
\end{thm}
\begin{proof}
	Suppose that $\rsf F$ is untwisted relative to $\rsf G$. Let 
	$\phi\colon\mathcal L\to \rsf I(B_{\rsf G})$ be the natural 
	embedding induced by the $\rsf G$-presentation of $\mathcal L$. 
	
	First we recall the definition of $\phi$ -- 
	$\mathcal L\simeq\rsf I({\rsf G})$ by 
	\begin{equation}
		x\mapsto\brk<\Delta(\one, x)\join\beta_{\rsf G}(x), 
		x\join\beta_{\rsf G}(x)>. 
		\label{eq:oneA}
	\end{equation}
	And $\rsf G$ embeds into $B_{\rsf G}$ by $g\mapsto\brk<g, 1>$, so 
	that 
	\begin{equation}
		\phi(x)=\brk<{\brk<\Delta(\one, x)\join\beta_{\rsf G}(x), 1>}, 
		{\brk<x\join\beta_{\rsf G}(x), 1>}>
		\label{eq:twoA}
	\end{equation}
	
	Atoms in $\rsf I(B_{\rsf G})$ are of the form $\brk<\comp a, a>$ 
	for $a\in B_{\rsf G}$. Thus $\phi[\rsf F]$ is untwisted iff there 
	is some $a\in B_{\rsf G}$ such that
	\begin{equation}
		\brk<\comp a, a>\le\phi(f)\qquad\text{ for all }f\in\rsf F. 
		\label{eq:three}
	\end{equation}
	
	Now such an $a$ is either $\brk<g, 1>$ or $\brk<g, 0>$ for some 
	$g\in\rsf G$ -- so we can rewrite \eqref{eq:three} as
	\begin{subequations}
		\begin{align}
			\brk<g, 0> & \le\brk<\Delta(\one, f)\join\beta_{\rsf G}(f), 1> 
			\label{eq:fourA}  \\
			\brk<g, 1> & \le\brk<f\join\beta_{\rsf G}(f), 1> \label{eq:fourB} 
		\end{align}
	\end{subequations}
	or
	\begin{subequations}
		\begin{align}
			\brk<g, 1> & \le\brk<\Delta(\one, f)\join\beta_{\rsf G}(f), 1> 
			\label{eq:fiveA}  \\
			\brk<g, 0> & \le\brk<f\join\beta_{\rsf G}(f), 1> \label{eq:fiveB} 
		\end{align}
	\end{subequations}
	We will only look at equations (4) -- which are assumed to hold for 
	all $f\in\rsf F$. 
	
	First we recall from \lemref{lem:interTwoFilAA} that $\rsf F\cap\rsf G=\Set{x\join\beta_{\rsf G}(x) | 
	x\in\rsf F}$. Thus \eqref{eq:fourB} implies $g\le z$ for all $z\in\rsf F\cap\rsf G$. 
	
	Now
	\eqref{eq:fourA} holds iff 
	\begin{align*}
		 & \brk<g, 0> \le\brk<\Delta(\one, f)\join\beta_{\rsf G}(f), 1>  \\
		\text{ iff } & \brk<g, 0> \join\brk<\Delta(\one, f)\join\beta_{\rsf G}(f), 1> =
		\brk<\Delta(\one, f)\join\beta_{\rsf G}(f), 1> \\
		\text{ iff } & \brk<g, 1> \to\brk<\Delta(\one, f)\join\beta_{\rsf G}(f), 1> =
		\brk<\Delta(\one, f)\join\beta_{\rsf G}(f), 1>  \\
		\text{ iff } & g \to(\Delta(\one, f)\join\beta_{\rsf G}(f))=
		\Delta(\one, f)\join\beta_{\rsf G}(f). 
	\end{align*}
	As this holds for all $f\in\rsf F$ it also holds for $g'=\beta_{\rsf 
	F}(g)$ and in this case we have 
	$g\to(\Delta(\one, g')\join g)=\Delta(\one, g')\join g$. 
	As $g\le \Delta(\one, g')\join g$ this entails 
	$\Delta(\one, g')\join g=\one$. Since $\mathcal L$ is an MR-algebra 
	we then have $g'\meet g$ exists. But $g'\simeq g$ so we have $g'=g$. 
	
	Hence $[g, \one]=\rsf F\cap\rsf G$. 
	
	The version of this argument using equations (5) is much the same, 
	except we have $\Delta(\one, g')=g$ at the end, making 
	$\rsf G\cap\Delta(\one, \rsf F)$ principal. 
	
	Conversely,  suppose that
	$\rsf F\cap\rsf G=[g, \one]$ is principal. 
	
	Then clearly we have $\brk<g, 1>\le\brk<f\join\beta_{\rsf G}(f), 1>$ 
	for all $f\in\rsf F$. 
	
	Also, for any $f\in\rsf F$ the meet $g\meet f$ exists so that
	$g\join\Delta(\one, f)=\one$. Hence 
	$g\join\Delta(\one, f)\join\beta_{\rsf G}(f)=1$ and so
	$g\to(\Delta(\one, f)\join\beta_{\rsf G}(f))=\Delta(\one, f)\join\beta_{\rsf G}(f)$. 
	Thus 
	$\brk<g, 0>\le\brk<\Delta(\one, f)\join\beta_{\rsf G}(f), 1>$. 
	
	Hence $\brk<{\brk<g, 0>}, {\brk<g, 1>}>\le\phi(f)$ for all $f\in\rsf F$. 
	
	If $\Delta(\one, \rsf F)\cap\rsf G=[g, \one]$ is principal then we 
	obtain
	$\brk<{\brk<g, 1>}, {\brk<g, 0>}>\le\phi(f)$ for all $f\in\rsf F$. 
\end{proof}

\begin{cor}\label{cor:relTwistingSymm}
	Let $\rsf F$ and $\rsf G$ be two g-filters for $\mathcal L$. Then
	$\rsf F$ is untwisted relative to $\rsf G$ iff
	$\rsf G$ is untwisted relative to $\rsf F$. 
\end{cor}
\begin{proof}
	This follows directly from the theorem as 
	$\Delta(\one, \rsf F)\cap\rsf G$ is principal iff
	$\Delta(\one, \rsf G)\cap\rsf F$ is principal. 
\end{proof}

\begin{defn}\label{def:equivTwist}
	Let $\rsf F$ and $\rsf G$ be two g-filters for $\mathcal L$. Then
	$$
	\rsf F\sim\rsf G\text{ iff }\rsf F\text{ is untwisted relative 
	to }\rsf G. 
	$$
\end{defn}

So far we have that $\sim$ is reflexive and symmetric. Now we will 
show that it is also transitive. 

Let $\rsf F, \rsf G$ and $\rsf H$ be three g-filters for $\mathcal 
L$ and suppose that both $\rsf F$ and $\rsf G$ are untwisted relative 
to $\rsf H$. 

\begin{lem}\label{lem:tranUntwist}
	Suppose that $a\le\phi[\rsf F]$ in $\rsf I(B_{\rsf H})$. Then the 
	diagram
	\begin{diagram}
		\rsf F & \rTo^{\phi} & \rsf I(B_{\rsf H})  \\
		\dTo^{\beta_{\rsf H}} &  & \dTo_{f_{a\zero}}  \\
		\rsf H & \rTo_{\phi} & \rsf I(B_{\rsf H})
	\end{diagram}
	commutes. 
\end{lem}
\begin{proof}
	First we check how $\phi$ acts upon $\rsf H$. 
	We have, for any $h\in\rsf H$
	\begin{diagram}
		h & \rMapsto & \brk<\Delta(\one, h)\join h, h\join h>=\brk<\one, 
		h> & \rMapsto & \brk<{\brk<1, 1>}, {\brk<h, 1>}>.   
	\end{diagram}
	We also recall that
	$$
	\begin{array}{c!{\text{ is }}l}
		0\text{ in }B_{\rsf H} & \brk<1, 0>  \\
		1\text{ in }B_{\rsf H} & \brk<1, 1>  \\
		\zero\text{ in }\rsf I(B_{\rsf H}) &  \brk<{\brk<1, 1>}, {\brk<1, 0>}> \\
		\one\text{ in }\rsf I(B_{\rsf H}) &  \brk<{\brk<1, 1>}, {\brk<1, 1>}>. 
	\end{array}
	$$
	Now let $f\in\rsf F$. Then
	\begin{align*}
		\phi(f) & =\brk<{\brk<\Delta(\one, f)\join\beta_{\rsf H}(f), 1>}, 
		{\brk<f\join\beta_{\rsf H}(f), 1>}>  \\
		f_{a\zero}(\phi(f)) & = (\phi(f)\join\zero)\meet(\Delta(\one, 
		\phi(f))\join\zero) \\
		\phi(f)\join\zero & = \brk<{\brk<\Delta(\one, f)\join\beta_{\rsf H}(f), 1>}, 
		{\brk<f\join\beta_{\rsf H}(f), 1>}>\join 
		\brk<{\brk<1, 1>}, {\brk<1, 0>}> \\
		 & = \brk<{\brk<\Delta(\one, f)\join\beta_{\rsf H}(f), 
		 1>\join\brk<1, 1>}, 
		{\brk<f\join\beta_{\rsf H}(f), 1>\join\brk<1, 0>}> \\
		 & = \brk<{\brk<1, 1>}, 
		{\brk<f\join\beta_{\rsf H}(f), 1>}>.  \\
		\Delta(\one, 
		\phi(f))\join\zero & =
		\brk<{\brk<f\join\beta_{\rsf H}(f), 1>}, 
		{\brk<\Delta(\one, f)\join\beta_{\rsf H}(f), 1>}>\join 
		\brk<{\brk<1, 1>}, {\brk<1, 0>}> \\
		 &  =
		\brk<{\brk<f\join\beta_{\rsf H}(f), 1>\join\brk<1, 1>}, 
		{\brk<\Delta(\one, f)\join\beta_{\rsf H}(f), 1>\join\brk<1, 0>}>\\
		 & = \brk<{\brk<1, 1>}, 
		{\brk<\Delta(\one, f)\join\beta_{\rsf H}(f), 1>}>.  \\
		\intertext{Thus }
		f_{a\zero}(\phi(f)) & = (\phi(f)\join\zero)\meet(\Delta(\one, 
		\phi(f))\join\zero) \\
		 & =\brk<{\brk<1, 1>}, {\brk<f\join\beta_{\rsf H}(f), 1>}>\meet
		 \brk<{\brk<1, 1>}, {\brk<\Delta(\one, f)\join\beta_{\rsf H}(f), 
		 1>}>\\
		 &= \brk<{\brk<1, 1>}, {\brk<f\join\beta_{\rsf H}(f), 1>\meet 
		 \brk<\Delta(\one, f)\join\beta_{\rsf H}(f), 1>}>\\
		 &= \brk<{\brk<1, 1>}, {\brk<(f\join\beta_{\rsf H}(f))\meet
		 (\Delta(\one, f)\join\beta_{\rsf H}(f)), 1>}>\\
		 &= \brk<{\brk<1, 1>}, {\brk<\beta_{\rsf H}(f), 1>}>\\
		 &= \phi(\beta_{\rsf H}(f)). 
	\end{align*}
\end{proof}

\begin{cor}\label{cor:ultraFF}
	$\phi[\rsf F]$ is an ultrafilter in $[a, \one]$. 
\end{cor}
\begin{proof}
	Since we know that $\Set{\beta_{\rsf H}(f) | f\in\rsf F}=\rsf H$, 
	$\phi[\rsf H]$ is an ultrafilter in $[\zero, \one]$ and that
	$f_{\zero a}\colon[\zero, \one]\to[a, \one]$ is a Boolean isomorphism, 
	we know that 
	$\phi[\rsf F]= f_{\zero a}[\phi[\beta_{\rsf H}[\rsf F]]]= 
	f_{\zero a}[\phi[\rsf H]]$ is also an ultrafilter. 
\end{proof}

We have an implication isomorphism $\beta_{\rsf F}\colon\rsf H\to\rsf 
F$. This extends to a Boolean isomorphism 
$\psi\colon B_{\rsf H}\to B_{\rsf F}$ by $\brk<h, 
i>\mapsto\brk<\beta_{\rsf F}(h), i>$. 

\begin{lem}\label{lem:equalMappings}
	The mapping
	\begin{diagram}
		\rsf F & \rTo^{\phi_{\rsf H}} & \rsf I(B_{\rsf H})  &
		\rTo^{f_{a\zero}} & \rsf I(B_{\rsf H}) & \rTo^{\rsf I(\psi)} & \rsf 
		I(B_{\rsf F})
	\end{diagram}
	is equal to $\phi_{\rsf F}\colon\rsf F\to\rsf I(B_{\rsf F})$. 
\end{lem}
\begin{proof}
	Recall that $\phi_{\rsf F}(f)=\brk<{\brk<1, 1>}, {\brk<f, 1>}>$ for 
	all $f\in\rsf F$. So we have 
	\begin{diagram}
		\rsf F & \rTo^{\phi_{\rsf H}} & \rsf I(B_{\rsf H}) &  &   \\
		\dTo^{\beta_{\rsf H}} &  & \dTo_{f_{a\zero}} &  &   \\
		\rsf H & \rTo_{\phi_{\rsf H}} & \rsf I(B_{\rsf H}) & \rTo_{\rsf I(\psi)}
		& \rsf I(B_{\rsf F}) 
	\end{diagram}
	so that any $f\in\rsf F$ is sent to
	\begin{align*}
		\rsf I(\psi)\circ f_{a\zero}\circ\phi_{\rsf H}(f) & =
		\rsf I(\psi)\circ\phi_{\rsf H}\circ\beta_{\rsf H}(f)\\
		& =\rsf I(\psi)(\brk<{\brk<1, 1>}, {\brk<\beta_{\rsf H}(f), 1>}>)  \\
		 & =\brk<{\brk<1, 1>, \brk<\beta_{\rsf F}\beta_{\rsf H}(f), 1>}>  \\
		 & =\brk<{\brk<1, 1>, \brk<f, 1>}>  \\
		 & =\phi_{\rsf F}(f). 
	\end{align*}
\end{proof}

\begin{thm}\label{thm:transTwist}
	Let $\rsf F, \rsf G$ and $\rsf H$ be three g-filters for $\mathcal 
	L$ and suppose that both $\rsf F$ and $\rsf G$ are untwisted relative 
	to $\rsf H$. Then $\rsf G$ is untwisted relative to $\rsf F$. 
\end{thm}
\begin{proof}
	Suppose that $b\in\rsf I(B_{\rsf H})$ is such that
	$b\le\phi[\rsf G]$. Then we have 
	\begin{align*}
		\phi_{\rsf F}[\rsf G] & = \rsf I(\psi)\circ f_{a\zero}\circ\phi_{\rsf 
		H}[\rsf G] \\
		 & \geq \rsf I(\psi)\circ f_{a\zero}(b). 
	\end{align*}
\end{proof}

\begin{cor}\label{cor:transTwist}
	Let $\rsf F$ be untwisted relative to $\rsf G$ and 
	$\rsf G$ be untwisted relative to $\rsf H$. Then
	$\rsf F$ is untwisted relative to $\rsf H$.
\end{cor}
\begin{proof}
	Since both $\rsf F$ and $\rsf H$ are untwisted relative to $\rsf G$ 
	we can apply the theorem to get the result. 
\end{proof}

\section{Connecting filters}
In this section we want to show that filters that are untwisted 
relative to one another satisfy another rather simple relation 
definable from $\Delta$. This gives us an easy method of producing all g-filters that 
are untwisted relative to some fixed g-filter $\rsf F$. 

Let $\rsf F$ be a filter in an MR-algebra $\mathcal L$ and let
$g\in\mathcal L$. 
\begin{lem}\label{lem:gFF}
	The set
	$$
	\rsf F_{g}=\Set{\Delta(g\join f, f) | f\in\rsf F}
	$$
	has fip and is upwards closed. 
\end{lem}
\begin{proof}
	We just need to check this for intervals. Suppose that
	$g=[g, h]$, $f_{0}=[x, y]\in\rsf F$ and $f_{1}=[s, t]\in\rsf F$. Then
	$$
	\Delta(g\join f_{1}, f_{1})=[(g_{0}\meet x)\join(g_{1}\meet\comp y), 
	(g_{0}\join\comp x)\meet(g_{1}\join y)]. 
	$$
	Thus 
	\begin{align*}
		 & \Delta(g\join f_{1}, f_{1})\meet\Delta(g\join f_{2}, f_{2})\text{ 
		 exists}  \\
		\text{iff } & (g_{0}\meet x)\join(g_{1}\meet\comp y)\le 
	(g_{0}\join\comp s)\meet(g_{1}\join t)  \\
		\text{ and } & (g_{0}\meet s)\join(g_{1}\meet\comp t)\le 
	(g_{0}\join\comp x)\meet(g_{1}\join y).  \\
	\intertext{We will only check the second inequality -- the other 
	follows by symmetry. }
		\text{Clearly } & (g_{0}\meet s)\join(g_{1}\meet\comp t)\le g_{1}\join y  \\
		\text{ and } & g_{0}\meet s\le (g_{0}\join\comp x)\meet(g_{1}\join 
		y) && \text{ as }g_{0}\le g_{1}.   \\
		 g_{1}\meet\comp t \le 
	 g_{0}\join\comp x & \text{ as }g_{1}\meet\comp t \le \comp 
	 t\le\comp x\le 
	 g_{0}\join\comp x. 
	\end{align*}
	
	To show upwards closure we note that if $k\geq\Delta(g\join f, f)$ 
	for some $f\in\rsf F$ then we have $k\in\leftGen\rsf F\rightGen$ and 
	so there is some $k'\in\rsf F$ with $k\simeq k'$. Then we have 
	$\Delta(g\join k', k')\simeq k$ and 
	(as above) $\Delta(g\join k', k')\meet k\geq \Delta(g\join k', 
	k')\meet \Delta(g\join f, f)$. Hence $k=\Delta(g\join k', k')$. 
\end{proof}

\begin{cor}\label{cor:gFF}
	The set
	$$
	\rsf F_{g}=\Set{\Delta(g\join f, f) | f\in\rsf F}
	$$
	is a filter and $\leftGen\rsf F_{g}\rightGen=\leftGen\rsf F\rightGen$. 
\end{cor}
\begin{proof}
	That $\rsf F_{g}$ is a filter follows from the lemma -- since if 
	$f_{1}, f_{2}\in\rsf F$ then we have $\Delta(g\join(f_{1}\meet f_{2}), 
	f_{1}\meet f_{2})
	\preccurlyeq f_{1}\meet f_{2}\preccurlyeq \Delta(g\join f_{i}, f_{i})$. 
	As the meet 
	$\Delta(g\join(f_{1}\meet f_{2}), f_{1}\meet f_{2})\meet 
	\Delta(g\join f_{i}, f_{i})$ exists we have 
	$\Delta(g\join(f_{1}\meet f_{2}), f_{1}\meet f_{2})\le 
	\Delta(g\join f_{i}, f_{i})$ and so 
	$\Delta(g\join(f_{1}\meet f_{2}), f_{1}\meet f_{2})\le 
	\Delta(g\join f_{1}, f_{1})\meet \Delta(g\join f_{2}, f_{2})$. 
	By upwards closure we then get 
	$\Delta(g\join f_{1}, f_{1})\meet \Delta(g\join f_{2}, f_{2})$ in 
	$\rsf F_{g}$. 
	
	By definition, 
	for each $f\in\rsf F$ there is a $f'\in\rsf F_{g}$ such that
	$f\simeq f'$, and conversely. Thus 
	$\leftGen\rsf F_{g}\rightGen=\leftGen\rsf F\rightGen$.
\end{proof}

Note that a special case of this is when $g=\one$ and we have
$\rsf F_{\one}=\Delta(\one, \rsf F)$ and that for a principal filter
$[h, \one]$ we have $[{h, \one}]_{g}=[\Delta(g\join h, h), \one]$. 

\begin{cor}\label{cor:extraFil}
	Let $\rsf F$ be a filter and $g\in\mathcal L$. Then
	$g\to\rsf F=\Set{g\to f | f\in\rsf F}$ is a filter. 
\end{cor}
\begin{proof}
	Recall that $g\to f=\Delta(\one, \Delta(g\join f, f))\join f$. 
	Hence if
	$\rsf G=\Delta(\one, \rsf F_{g})$ then
	\begin{align*}
		\rsf G\cap\rsf F & =\Set{f\join\beta_{\rsf G}(f) | f\in\rsf F}  \\
		 & =\Set{\Delta(\one, \Delta(g\join f, f))\join f | f\in\rsf F}  \\
		 & =\Set{g\to f | f\in\rsf F}. 
	\end{align*}
\end{proof}

\begin{lem}\label{lem:interSect}
	$$\rsf F\cap\rsf F_{g}=[g, \one]\cap\rsf F. $$
\end{lem}
\begin{proof}
	If $f\in\rsf F\cap[g, \one]$ then $g\join f=f$ and so
	$\Delta(g\join f, f)=\Delta(f, f)= f\in\rsf F_{g}$. 
	
	Conversely, if $h\in\rsf F\cap\rsf F_{g}$ then we have 
	$h\meet\Delta(g\join h, h)$ exists and so $h= \Delta(g\join h, h)$. 
	Therefore $g\join h=h$ and $g\le h$. 
\end{proof}

\begin{defn}\label{def:principalizable}
	Let $\mathcal L$ be a cubic algebra, $\rsf F$ a filter in $\mathcal L$. 
	Then $\rsf F$ is \emph{weakly principal} iff there is some $g$ such 
	that
	$\rsf F\subseteq[g, \one]$. 
\end{defn}

This definition is a lot like $\rsf F$ being untwisted. 

\begin{cor}\label{cor:gFilter}
	If $g\in\rsf F$ then 
	$$\rsf F\cap\rsf F_{g}=[g, \one]. $$
\end{cor}
\begin{proof}
	Obvious
\end{proof}

Interestingly enough the converse of this result is also true. 
\begin{lem}\label{lem:gFilter}
	Suppose that $\leftGen\rsf G\rightGen=\leftGen\rsf F\rightGen$
	and $\rsf F\cap\rsf G=[g, \one]$. Then
	$$\rsf G=\rsf F_{g}.$$ 
\end{lem}
\begin{proof}
	Clearly $[g, \one]\subseteq\rsf G$. 
	Suppose that
	$h\le g$ and $h\in\rsf G$. Then we have $h'\in\rsf F$ such that
	$h\simeq h'$. As $g, h'$ are in $\rsf F$ the meet $g\meet h'$ exists 
	and $h\le g$ implies $h'\preccurlyeq g$. Hence $h'\le g$. 
	But then $h'\join h\in\rsf F\cap\rsf G=[g, \one]$ so we have 
	$h\join h'=g$. 
	Therefore 
	$h= \Delta(h\join h', h')= \Delta(g\join h', h')\in\rsf F_{g}$. 
	
	Now for arbitrary $h\in\rsf G$ we have $h\geq h\meet g\in\rsf F_{g}$ 
	and so $h\in\rsf F_{g}$. Thus $\rsf G\subseteq\rsf F_{g}$. 
	
	The reverse implication follows as 
	$\leftGen\rsf G\rightGen= \leftGen\rsf F\rightGen= \leftGen\rsf F_{g}\rightGen$
	and so if $h\in\rsf F_{g}$ there is some $h'\in\rsf G$ with $h\simeq 
	h'$. As $h\meet h'$ exists we have $h=h'\in\rsf G$. 
\end{proof}

\begin{cor}\label{cor:Idempotence}
	Let $g, h\in\rsf F$. Then 
	\begin{enumerate}[(a)]
		\item $\rsf F=(\rsf F_{g})_{g}$; 
	
		\item $(\rsf F_{g})_{h}=(\rsf F_{g})_{g\join h}$. 
	\end{enumerate}
\end{cor}
\begin{proof}
	\begin{enumerate}[(a)]
		\item Since $\rsf F\cap\rsf F_{g}=[g, \one]$and 
		$\leftGen\rsf F\rightGen= \leftGen\rsf F_{g}\rightGen$ the lemma 
		implies $\rsf F=(\rsf F_{g})_{g}$. 
	
		\item 
		\begin{align*}
			\rsf F_{g}\cap(\rsf F_{g})_{h} & = [h, \one]\cap\rsf F_{g}  \\
			 &  =[h, \one]\cap\rsf F\cap\rsf F_{g} \\
			 &  =[h, \one]\cap[g, \one] \\
			 &  =[h\join g, \one]. 
		\end{align*}
		The lemma now implies $(\rsf F_{g})_{h}=(\rsf F_{g})_{g\join h}$. 
	\end{enumerate}
\end{proof}

It is also interesting to determine the shape of $\rsf 
F\cap\Delta(\one, \rsf F_{g})$. Corollary \ref{cor:extraFil} already 
shows us that this is 
$g\to\rsf F= \Set{g\to f | f\in\rsf F}$. 
Lemma \ref{lem:gFilter} shows us that if this is principal ($=[h, \one]$) then it is 
contained in $\rsf F_{h}$.

\section{The Group of g-filters}
The set of all g-filters of a filter algebra $\mathcal L$ is a group -- 
where to define the operation we first fix a g-filter $\rsf F$ and let
\begin{align*}
	\rsf G*\rsf H
	&=\varphi_{\brk<\rsf F, \rsf G>}[\rsf H]\\
	&=\Set{\Delta(\beta_{\rsf G}(\alpha_{\rsf F}(h)), \beta_{\rsf 
	G}(h)) | h\in H}. 
\end{align*}

This group is a $2$-torsion group that has a natural topology -- 
the basic opens are of the form $\mathcal O_{m}=\Set{\rsf H | 
m\in\rsf H}$ for each $m\in\mathcal L$.  
We are interested in determining 
whether $\sim$ is a congruence on this group. 

It is clearly so in the case that $\mathcal L\simeq\rsf I(B)$ for some Boolean 
algebra -- as in that case all g-filters are of the form
$[a, \one]$ for some vertex $a$, and are all equivalent to one 
another and so there is only one equivalence class. 

\begin{lem}\label{lem:congOne}
	Let $\rsf H_{1}\cap\rsf H_{2}=[h, \one]$. Then 
	$(\rsf G*\rsf H_{1})\cap (\rsf G*\rsf H_{2})$ is principal. 
\end{lem}
\begin{proof}
	\begin{align*}
		(\rsf G*\rsf H_{1})\cap (\rsf G*\rsf H_{2}) & =\varphi_{\brk<\rsf F, \rsf G>}
		[\rsf H_{1}]\cap \varphi_{\brk<\rsf F, \rsf G>}[\rsf H_{2}] \\
		 & =\varphi_{\brk<\rsf F, \rsf G>}[\rsf H_{1}\cap\rsf H_{2}]  \\
		 & =\varphi_{\brk<\rsf F, \rsf G>}[[\rsf H_{1}\cap\rsf H_{2}]h, \one]  \\
		 & =[\varphi_{\brk<\rsf F, \rsf G>}(h), \one]. 
	\end{align*}
\end{proof}

\begin{cor}\label{cor:congOne}
	Let $\rsf H_{1}\sim\rsf H_{2}$. Then 
	$(\rsf G*\rsf H_{1})\sim (\rsf G*\rsf H_{2})$. 
\end{cor}
\begin{proof}
	If $\rsf H_{1}\cap\rsf H_{2}$ is principal we are done. 
	
	If $\rsf H_{1}\cap\Delta(\one, \rsf H_{2})=[h, \one]$ then 
	the result follows as above using 
	$\varphi_{\brk<\rsf F, \rsf G>}[\Delta(\one, \rsf H_{2})]=
	\Delta(\one, \varphi_{\brk<\rsf F, \rsf G>}[\rsf H_{2}])$. 
\end{proof}

\begin{thm}\label{thm:congOne}
	$\sim$ is a congruence relation. 
\end{thm}
\begin{proof}
	As if $\rsf G_{1}\sim\rsf G_{2}$ and $\rsf H_{1}\sim\rsf H_{2}$
	then
	$\rsf G_{1}*\rsf H_{1}\sim \rsf G_{1}*\rsf H_{2}= \rsf H_{2}*\rsf 
	G_{1}\sim
	\rsf H_{2}*\rsf G_{2}= \rsf G_{2}*\rsf H_{2}$. 
\end{proof}

Lastly we observe that any equivalence class is dense in the group -- 
since if $m\in\mathcal L$ and $\rsf G$ is any g-filter then we have 
$g= m\join\beta_{\rsf G}(m)\in\rsf G$ and
so $m= \Delta(m\join\beta_{\rsf G}(m), \beta_{\rsf G}(m))=
\Delta(g\join \beta_{\rsf G}(m), \beta_{\rsf G}(m))\in\rsf G_{g}$. 
As $\rsf G\sim\rsf G_{g}\in\mathcal O_{m}$. Thus we see that the 
equivalence class of $\rsf G$  hits every basic open. 

\section{Extending Everything to Filters}\label{sect:seven}
Let $\mathcal L$ be a filter algebra. The construction of $\rsf 
F_{g}$ is in some sense a generalization of $\Delta$ to the set of 
g-filters of $\mathcal L$. In this section we seek an expansion of 
this operation to all filters. 

We recall that $\Delta$ and implication are closely related in cubic 
algebras. The approach we take here is to define implication first 
and use it to define $\Delta$. For reasons that may become clear 
later we will use the reverse order on filters in this section. 

\subsection{Relative Complements}\label{subsec.relcomp}
Let
$\rsf G\subseteq\rsf F$ be two $\mathcal L$-filters. There are 
several ways to define the relative complement of $\rsf G$ in $\rsf F$. 

\begin{defn}\label{def:impl}
	Let $\rsf G\subseteq\rsf F$ be two $\mathcal L$-filters. Then
	\begin{enumerate}[(a)]
		\item $\rsf G\supset\rsf F=\bigcap\Set{\rsf H | \rsf H\join\rsf G=\rsf F}$; 
	
		\item $\rsf G\Rightarrow\rsf F=\bigvee\Set{\rsf H |\rsf H\subseteq\rsf 
		F\text{ and }\rsf H\cap\rsf G={\Set{\one}}}$; 
	
		\item $\rsf G\to\rsf F=\Set{h\in\rsf F |\forall g\in\rsf G\ h\join 
		g=\one}$. 
	\end{enumerate}
\end{defn}

We will now show that these all define the same set. 

\begin{lem}\label{lem:twoThreeSame}
	$\rsf G\to\rsf F=\rsf G\Rightarrow\rsf F$. 
\end{lem}
\begin{proof}
	Let $h\in(\rsf G\to\rsf F)\cap\rsf G$. Then $\one=h\join h=h$. Thus
	$\rsf G\to\rsf F\subseteq\rsf G\Rightarrow\rsf F$. 
	
	Suppose that $\rsf H\subseteq\rsf F$ and $\rsf H\cap\rsf G=\Set{\one}$. 
	Let $h\in\rsf H$ and $g\in\rsf G$. Then
	$h\join g\in\rsf H\cap\rsf G=\Set{\one}$ so that $h\join g=\one$. 
	Hence $\rsf H\subseteq(\rsf G\to\rsf F)$ and so
	$\rsf G\Rightarrow\rsf F\subseteq\rsf G\to\rsf F$. 
\end{proof}

\begin{lem}\label{lem:smallH}
	Let $h\in\rsf F$ and $g\in\rsf G$ be such that $g\join h<\one$. Then
	$h\notin g\to\rsf F$. 
\end{lem}
\begin{proof}
	This is clear as $h=g\to f$ implies $h\join g=\one$. 
\end{proof}

%
%
%

\begin{thm}\label{thm:oneThreeEqual}
	$\rsf G\supset\rsf F=\rsf G\to\rsf F$. 
\end{thm}
\begin{proof}
	Suppose that $h\notin \rsf G\to\rsf F$ so that there is some 
	$g\in\rsf G$ with $h\join g<\one$. Then $h\notin g\to\rsf F$ and 
	clearly $\rsf F=[g, \one]\join(g\to\rsf F)$ so that $\rsf 
	G\supset\rsf F\subseteq g\to\rsf F$ does not contain $h$. Thus
	$\rsf G\supset\rsf F\subseteq\rsf G\to\rsf F$. 
	
	Conversely if $\rsf H\join\rsf G=\rsf F$ and $k\in\rsf G\to\rsf F$ 
	then there is some $h\in\rsf H$ and $g\in\rsf G$ with $k=h\meet g$. 
	But then 
	\begin{align*}
		k & =k\join(h\meet g)  \\
		 & =(k\join h)\meet(k\join g)  \\
		 & =k\join h && \text{ as }k\join g=\one
	\end{align*}
	and so $k\geq h$ must be in $\rsf H$. 
	Thus $\rsf G\to\rsf F \subseteq\rsf G\supset\rsf F$. 
\end{proof}

We earlier defined a filter $g\to\rsf F$. We now show that this new 
definition of $\to$ extends this earlier definition. 

\begin{lem}\label{lem:implGG}
	Let $g\in\rsf F$. Then 
	$$
	g\to\rsf F=[g, \one]\to\rsf F. 
	$$
\end{lem}
\begin{proof}
	Let $g\to f\in g\to\rsf F$ and $k\in[g, \one]$. Then
	$k\join(g\to f)\geq g\join (g\to f)=\one$. Thus $g\to f\in[g, 
	\one]\to\rsf F$ and so $g\to \rsf F\subseteq[g, 
	\one]\to\rsf F$. 
	
	Conversely, if $h\in[g, \one]\to\rsf F$ then $h\join g=\one$ and so
	$h$ is the complement of $g$ in $[h\meet g, \one]$. Thus
	$h=g\to(h\meet g)\in g\to\rsf F$ and so $[g,\one]\to\rsf F\subseteq
	g\to \rsf F$. 
\end{proof}

\begin{lem}\label{lem:inclImpl}
	Let $\rsf G\subseteq\rsf H\subseteq\rsf F$. Then
	$$
	\rsf G\to\rsf H\subseteq\rsf G\to\rsf F. 
	$$
\end{lem}
\begin{proof}
	If $h\in\rsf H$ and $h\join g=\one$ for all $g\in\rsf G$ then
	$h\in\rsf G\to\rsf F$. 
\end{proof}

\begin{cor}\label{cor:inclImpl}
	Let $\rsf G\subseteq\rsf H\subseteq\rsf F$ and
	$\rsf G\to\rsf F\subseteq\rsf H$. Then
	$$
	\rsf G\to\rsf H=\rsf G\to\rsf F. 
	$$
\end{cor}
\begin{proof}
	LHS$\subseteq$RHS by the lemma. Conversely if $h\in\rsf G\to\rsf F$ 
	then 
	$h\in\rsf H$ has the defining property for $\rsf G\to\rsf H$ and so 
	is in $\rsf G\to\rsf H$. 
\end{proof}

\begin{cor}\label{cor:conJoint}
	$$
	\rsf G\to(\rsf G\join(\rsf G\to\rsf F))=\rsf G\to\rsf F. 
	$$
\end{cor}

\begin{lem}\label{lem:inclAgain}
	Let $\rsf G\subseteq\rsf H\subseteq\rsf F$. Then
	$$
	\rsf H\to\rsf F\subseteq\rsf G\to\rsf F. 
	$$
\end{lem}
\begin{proof}
	This is clear as $k\join h=\one$ for all $h\in\rsf H$ implies
	$k\join g=\one$ for all $g\in\rsf G$. 
\end{proof}

\subsection{Delta on Filters}

Now the critical lemma in defining our new $\Delta$ operation. 

\begin{lem}\label{lem:deltaOne}
	$(\rsf G\to\rsf F)\cup\Delta(\one, \rsf G)$ has fip. 
\end{lem}
\begin{proof}
	If $x\in \rsf G\to\rsf F$ and $y\in\rsf G$ then 
	$x\join y=\one$  and so (as $\mathcal L$ is an MR-algebra) we know 
	that $x\meet\Delta(\one, y)$ exists. 
\end{proof}

\begin{defn}\label{def:Delta}
	Let $\rsf G\subseteq\rsf F$. Then 
	$$
	\Delta(\rsf G, \rsf F)=\Delta(\one, \rsf G\to\rsf F)\join\rsf G. 
	$$
\end{defn}

The simplest filters in $\rsf F$ are the principal ones. In this case 
we obtain the following result. 

\begin{lem}\label{lem:DeltagOne}
	Let $g\in\rsf F$. Then $\Delta([g, \one], \rsf F)=\rsf F_{g}$. 
\end{lem}
\begin{proof}
	From \lemref{lem:implGG} we have $[g, \one]\to\rsf F=g\to\rsf F$ and 
	we know from \corref{cor:extraFil} that
	$\Delta(\one, \rsf F_{g})\cap\rsf F=g\to\rsf F$. Thus
	$\Delta(\one, g\to\rsf F)\subseteq\rsf F_{g}$. 
	Also $g\in\rsf F_{g}$ so we have 
	$\Delta([g, \one], \rsf F)\subseteq\rsf F_{g}$. 
	
	Conversely, if $f\in\rsf F$ then 
	$\Delta(g\join f, f)=(g\join f)\meet\Delta(\one, g\to f)$ is in 
	$\Delta(1, g\to\rsf F)\join[g, \one]=\Delta([g, \one], \rsf F)$. 
\end{proof}

\begin{cor}\label{cor:DeltaDoublePrinc}
	Let $g\geq h$ is $\rsf F$. Then 
	$$
		\Delta([g, \one], [h, \one])=[\Delta(g, h), \one]. 	
	$$
\end{cor}
\begin{proof}
	As $\Delta([g, \one], [h, \one])= [h, \one]_{g}= [\Delta(g, h), \one]$. 
\end{proof}

For further properties of the $\Delta$ operation we need some 
facts about the interaction between $\to$ and $\Delta$. Here is the first. 

\begin{lem}\label{lem:implInDelta}
	$$
	\rsf G\to\Delta(\rsf G, \rsf F)=\Delta(\one, \rsf G\to\rsf F). 
	$$
\end{lem}
\begin{proof}
	Let $k\in\rsf G\to\rsf F$ and $h=\Delta(\one, k)$. Then
	$k\meet g$ exists so $\Delta(\one, h)\meet g=k\meet g$ exists and 
	therefore $h\join g=\one$. Hence $\Delta(\one, \rsf G\to\rsf 
	F)\subseteq \rsf G\to\Delta(\rsf G, \rsf F)$. 
	
	Conversely, suppose that $h\in\Delta(\rsf G, \rsf F)$ and for all $g\in\rsf G$ 
	we have $h\join g=\one$. Then
	there is some $k\in\rsf G\to\rsf H$ and $g'\in\rsf G$ such that
	$h=\Delta(\one, k)\meet g'$. Then 
	$\one = h\join g'= (\Delta(\one, k)\meet g')\join g'= g'$. Thus
	$h=\Delta(\one, k)\in \Delta(\one, \rsf G\to\rsf F)$. 
\end{proof}

\begin{cor}\label{cor:doubleDelta}
	$$
	\Delta(\rsf G, \Delta(\rsf G, \rsf F))=\rsf G\join(\rsf G\to\rsf F). 
	$$
\end{cor}
\begin{proof}
	\begin{align*}
		\Delta(\rsf G, \Delta(\rsf G, \rsf F)) & \Delta(\one, 
		\rsf G\to\Delta(\rsf G, \rsf F))\join\rsf G  \\
		 & =\Delta(\one, \Delta(\one, \rsf G\to\rsf F))\join\rsf G  \\
		 & =(\rsf G\to\rsf F)\join\rsf G. 
	\end{align*}
\end{proof}

\begin{lem}\label{lem:inclDelta}
	Let $\rsf G\subseteq\rsf H\subseteq\rsf F$. Then
	$$
	\Delta(\rsf G, \rsf H)\subseteq\Delta(\rsf G, \rsf F). 
	$$
\end{lem}
\begin{proof}
	As $\Delta(\one, \rsf G\to\rsf H)\join\rsf G\subseteq
	\Delta(\one, \rsf G\to\rsf F)\join\rsf G$. 
\end{proof}

\begin{lem}\label{lem:interDelta}
	$\rsf F\cap\Delta(\rsf G, \rsf F)=\rsf G$. 
\end{lem}
\begin{proof}
	Clearly $\rsf G\subseteq \rsf F\cap\Delta(\rsf G, \rsf F)$. 
	
	Let $g\in\rsf G$ and $k\in\rsf G\to\rsf F$ be such that
	$f=g\meet\Delta(\one, k)\in\rsf F$. Then $k\in\rsf F$ so the meet
	$k\meet\Delta(\one, k)$ exists. Thus $k=\one$ and so $f=g\in\rsf G$. 
\end{proof}

\subsection{Boolean elements}
Corollary \ref{cor:doubleDelta} shows us what happens to $\Delta(\rsf G, 
\Delta(\rsf G, \rsf F))$. We are interested in knowing when this produces 
$\rsf F$. 

\begin{defn}\label{def:Boolean}
	Let $\rsf F$ be a g-filter. Then
	\begin{enumerate}[(a)]
		\item $\rsf G$ is \emph{weakly $\rsf F$-Boolean} iff $\rsf G\subseteq\rsf 
		F$ and
		$(\rsf G\to\rsf F)\to\rsf F=\rsf G$. 
	
		\item $\rsf G$ is \emph{weakly Boolean} iff 
		there is some g-filter containing $\rsf G$ and 
		$\rsf G$ is $\rsf H$-Boolean for all such g-filters $\rsf H$.
		
		\item $\rsf G$ is \emph{$\rsf F$-Boolean} iff $\rsf G\subseteq\rsf 
		F$ and $\rsf G\join(\rsf G\to\rsf F)=\rsf F$. 
		
		\item $\rsf G$ is \emph{Boolean} iff
		there is some g-filter containing $\rsf G$ and 
		$\rsf G$ is $\rsf H$-Boolean for all such g-filters $\rsf H$.
	\end{enumerate}
\end{defn}

The most interesting of these definitions is the last one. 
Before continuing however we show that ``weak'' really is weaker. 

\begin{lem}\label{lem:BooleanANotherWay}
	Suppose that $\rsf G$ is $\rsf F$-Boolean. 
	Then $\rsf G$ is weakly $\rsf F$-Boolean. 
\end{lem}
\begin{proof}
	We know that $\rsf G\subseteq(\rsf G\to\rsf F)\to\rsf F$. 
	
	Since $\rsf G\join(\rsf G\to\rsf F)=\rsf F$ we also have that
	$(\rsf G\to\rsf F)\supset\rsf F\subseteq\rsf G$. 
\end{proof}

And now the simplest examples of $\rsf F$-Boolean filters. 

\begin{lem}\label{lem:implgTwice}
	Let $g\in\rsf F$. Then 
	$[g, \one]$ is $\rsf F$-Boolean. 
\end{lem}
\begin{proof}
	We know that 
	$$
	\Delta([g, \one], \rsf F_{g})=(\rsf F_{g})_{g}=\rsf F
	$$
	and so 
	\begin{align*}
		\rsf F & =[g, \one]\join\Delta(\one, [g, \one]\to\rsf F_{g})  \\
		 & =[g, \one]\join \Delta(\one, [g, \one]\to\Delta([g, \one], \rsf F))  \\
		 & =[g, \one]\join\Delta(\one, \Delta(\one, [g, \one]\to\rsf F))\\
		 & = [g, \one]\join (g\to\rsf F). 
	\end{align*}
\end{proof}

Essentially because we have so many filter automorphisms we can show 
that Boolean is not a local concept -- that is if $\rsf G$ is Boolean 
somewhere then it is Boolean everywhere. And similarly for weakly 
Boolean. 

\begin{lem}\label{lem:moving}
	Let $\rsf F\sim\rsf H$ and $\rsf G\subseteq\rsf F\cap\rsf H$ be filters. 
	Let $\beta=\beta_{\rsf F\rsf H}\restrict\rsf F$ (and so
		$\beta^{-1}=\beta_{\rsf H\rsf F}\restrict\rsf H$). 
	Then $\beta[\rsf G\to\rsf F]=\beta[\rsf G]\to\rsf H$. 
\end{lem}
\begin{proof}
	Indeed if $g\in\rsf G$ and $h\in\rsf G\to\rsf F$ then we have 
		\begin{align*}
			\one & =\beta(h\join g)  \\
			 & =\beta(h)\join\beta(g)
		\end{align*}
		and so $\beta(h)\in\beta[\rsf G]\to\rsf H$. 
		
		Likewise, if $h\in\beta[\rsf G]\to\rsf H$ and $g\in\rsf G$ then
		$\one= h\join\beta(g)= \beta(\beta^{-1}(h)\join g)$ so that
		$\beta^{-1}(h)\join g=\one$. Thus $\beta^{-1}(h)\in\rsf G\to\rsf F$
		whence $h=\beta(\beta^{-1}(h))\in \beta[\rsf G\to\rsf F]$. 
\end{proof}

\begin{thm}\label{thm:Boolean}
	Let $\rsf G$ be $\rsf F$-Boolean for some g-filter $\rsf F$. 
	Then $\rsf G$ is Boolean. 
\end{thm}
\begin{proof}
	We have $\rsf G\join(\rsf G\to\rsf F)=\rsf F$ and $\rsf 
	G\subseteq\rsf H$. Let $h\in\rsf H$ and find 
	$g\in\rsf G$, $k\in\rsf G\to\rsf F$ with $\beta^{-1}(h)=g\meet k$. 
	Then
	$h= \beta(\beta^{-1}(h))= \beta(g\meet k)= \beta(g)\meet\beta(k)= 
	 g\meet\beta(k)$ as $g\in\rsf H$ implies $\beta(g)=g$. As 
	 $\beta(k)\in \beta[\rsf G\to\rsf F]= \beta[\rsf G]\to\rsf H= 
	 \rsf G\to\rsf H$ we have $h\in \rsf G\join(\rsf G\to\rsf H)$. 
\end{proof}

\begin{thm}\label{thm:wkBoolean}
	Let $\rsf G$ be weakly $\rsf F$-Boolean for some g-filter $\rsf F$. 
	Then $\rsf G$ is weakly Boolean. 
\end{thm}
\begin{proof}
	\begin{enumerate}[{Claim }1:]
		\item $\beta[\rsf G]=\rsf G$ -- since $\rsf G\subseteq\rsf H$ 
		implies $\beta\restrict\rsf G$ is the identity. 
	
		\item Now suppose that $\rsf G$ is $\rsf F$-Boolean. Then
		\begin{align*}
			\rsf G & =\beta[\rsf G]  \\
			 & =\beta[(\rsf G\to\rsf F)\to\rsf F]  \\
			 & =\beta[\rsf G\to\rsf F]\to\rsf H  \\
			 & =(\beta[\rsf G]\to\rsf H)\to\rsf H\\
			 & = (\rsf G\to\rsf H)\to\rsf H. 
		\end{align*}
	\end{enumerate}	
\end{proof}

We need to know certain persistence properties of Boolean-ness. 
\begin{lem}\label{lem:upwards}
	Let $\rsf G\subseteq\rsf H\subseteq\rsf F$ be $\rsf F$-Boolean. 
	Then $\rsf G$ is $\rsf H$-Boolean and $\rsf G\to\rsf H=(\rsf G\to\rsf 
	F)\cap\rsf H$. 
\end{lem}
\begin{proof}
	First we note that $\rsf G\to\rsf H=(\rsf G\to\rsf F)\cap\rsf H$ as 
	$x\in$LHS iff $x\in\rsf H$ and for all $g\in\rsf G$ $x\join 
	g=\one$ iff $x\in$RHS. 
	
	Thus we have 
	\begin{align*}
		\rsf H & =\rsf F\cap\rsf H  \\
		 & =(\rsf G\join(\rsf G\to\rsf F))\cap\rsf H  \\
		 & = (\rsf G\cap\rsf H)\join((\rsf G\to\rsf F)\cap\rsf H) \\
		 & =\rsf G\join(\rsf G\to\rsf H). 
	\end{align*}
\end{proof}

\begin{lem}\label{lem:middle}
	Let $\rsf G$ be $\rsf H$-Boolean, $\rsf H$ be $\rsf F$-Boolean. Then 
	$\rsf G$ is $\rsf F$-Boolean. 
\end{lem}
\begin{proof}
	Let $f\in\rsf F$. Then there is some $h\in\rsf H$ and $k\in\rsf 
	H\to\rsf F$ such that $h\meet k=f$. Also there is some 
	$g\in\rsf G$ and $l\in\rsf G\to\rsf H$ such that $h=g\meet l$. 
	Thus $g\meet l\meet k=f$ -- so it suffices to show that $l\meet 
	k\in\rsf G\to\rsf F$. 
	
	Clearly $k\meet l\in\rsf F$. So let $p\in\rsf G$. Then
	$\rsf G\subseteq\rsf H$ and $k\in\rsf H\to\rsf F$ implies $p\join 
	 k=\one$. $l\in\rsf G\to\rsf H$ implies $p\join l=\one$. 
	 Therefore
	 $p\join(k\meet l)=(p\join k)\meet(p\join l)= \one\meet\one= \one$. 
\end{proof}

So far we have few examples of Boolean filters. The next lemma 
produces many more. 

\begin{lem}\label{lem:lots}
	Let $\rsf F\sim\rsf H$. Then $\rsf F\cap\rsf H$ is $\rsf F$-Boolean 
	and
	$$
	(\rsf F\cap\rsf H)\to\rsf F=\Delta(\one, \rsf H)\cap\rsf F. 
	$$
\end{lem}
\begin{proof}
	First we show that $(\rsf F\cap\rsf H)\to\rsf F=\Delta(\one, \rsf H)\cap\rsf F$. 
	
	Let $f\in\rsf F\cap\rsf H$ and $k\in\Delta(\one, \rsf H)\cap\rsf F$. 
	Then $\Delta(\one, k)\in\rsf H$ so $\Delta(\one, k)\meet f$ exists. 
	Therefore $k\join f=\one$. Hence 
	$\Delta(\one, \rsf H)\cap\rsf F\subseteq (\rsf F\cap\rsf H)\to\rsf F$. 
	
	Conversely suppose that $k\in (\rsf F\cap\rsf H)\to\rsf F$. Let
	$h\in\rsf H$. Then $h\join k\in\rsf F\cap\rsf H$ and so
	$h\join k= (h\join k)\join k= \one$. Thus for all $h\in\rsf H$ we 
	have $h\meet\Delta(\one, k)$ exists. As there is some $k'\simeq k$ 
	in $\rsf H$ this implies $k'=\Delta(\one, k)\in\rsf H$ and so
	$k\in \Delta(\one, rsf H)\cap\rsf F$. 
	
	Now let $f\in\rsf F$. Then let $f'\in\rsf H$ with $f'\simeq f$. 
	Then 
	$(f\join f')\to f= f\meet \Delta(\one, \Delta(f'\join f, f))=
	f\meet \Delta(\one, f')\in \rsf F\cap\Delta(\one, \rsf H)$. 
	Also 
	$f\join f'\in\rsf F\cap\rsf H$ and $(f\join f')\meet ((f\join f')\to 
	f)=f$ so $f\in (\rsf F\cap\rsf H)\join(\rsf F\cap\Delta(\one, \rsf H))$. 
\end{proof}

\begin{cor}\label{cor:lots}
	Let $\rsf F\sim\rsf H$. Then 
	$$
	\Delta(\rsf F\cap\rsf H, \rsf F)=\rsf H. 
	$$
\end{cor}
\begin{proof}
	\begin{align*}
		\Delta(\rsf F\cap\rsf H, \rsf F)&=(\rsf F\cap\rsf H)\join
		\Delta(\one, (\rsf F\cap\rsf H)\to\rsf F)\\
		&=(\rsf F\cap\rsf H)\join\Delta(\one, \Delta(\one, \rsf H)\cap\rsf F)\\
		&=(\rsf F\cap\rsf H)\join(\rsf H\cap\Delta(\one, \rsf F))\\
		&=(\rsf F\cap\rsf H)\join((\rsf F\cap\rsf H)\to\rsf H)\\
		&=\rsf H
	\end{align*}
	since $\rsf F\cap\rsf H$ is also $\rsf H$-Boolean. 
\end{proof}

\begin{lem}\label{lem:DeltaInMR}
	Let $g, h$ in $\mathcal L$ be such that $g\meet h$ exists and 
	$g\join h=\one$. Then $\Delta(g, g\meet h)=g\meet\Delta(\one, h)$. 
\end{lem}
\begin{proof}
	\begin{align*}
		\Delta(g, g\meet h) & =g\meet\Delta(\one, g\to(g\meet h))  \\
		 & =g\meet \Delta(\one, (g\join h)\to h)  && \text{ by modularity in 
		 }[g\meet h, \one]\\ 
		 & = g\meet \Delta(\one, \one\to h) \\
		 & = g\meet \Delta(\one, h)
	\end{align*}
\end{proof}

\begin{thm}\label{thm:lots}
	$\rsf H\sim\rsf F$ iff there is an $\rsf F$-Boolean filter $\rsf G$ 
	such that $\rsf H=\Delta(\rsf G, \rsf F)$. 
\end{thm}
\begin{proof}
	The right to left direction is the last corollary. 
	
	So we want to prove that $\Delta(\rsf G, \rsf F)\sim\rsf F$ 
	whenever $\rsf G$ is $\rsf F$-Boolean. 
	
	Let $f\in\rsf F$. We will show that there is some $f'\in\Delta(\rsf G, 
	\rsf F)$ with $f\simeq f'$. 
	As $\rsf G\join(\rsf G\to\rsf F)=\rsf F$ we can find $g\in\rsf G$ and
	$h\in\rsf G\to\rsf F$ with $f=g\meet h$. As $g\join h=\one$ we know 
	that
	$\Delta(g, g\meet h)=g\meet\Delta(\one, h)$. But
	$g\meet\Delta(\one, h)\in \rsf G\join\Delta(\one, \rsf G\to\rsf H)= 
	\Delta(\rsf G, \rsf F)$ and $f=g\meet h\simeq \Delta(g, g\meet h)= g\meet\Delta(\one, h)$.¥
\end{proof}

The Boolean elements have nice properties with respect to $\Delta$. We 
want to show more -- that the set of $\rsf F$-Boolean elements is a 
Boolean subalgebra of $[\rsf F, \Set{\one}]$ with the reverse order. 
From this we will later show that the set of all Boolean filters is 
an atomic MR-algebra. 

It suffices to show closure under $\cap$ and $\join$ -- closure under 
$\to$ follows from \lemref{lem:BooleanANotherWay}. 

\begin{lem}\label{lem:interStrBool}
	Let $\rsf G_{1}$ and $\rsf G_{2}$ be $\rsf F$-Boolean. Then 
	$(\rsf G_{1}\to\rsf F)\join(\rsf G_{2}\to\rsf F)= (\rsf G_{1}\cap\rsf 
	G_{2})\to\rsf F$. 
\end{lem}
\begin{proof}
	Suppose that $h_{i}\in\rsf G_{i}\to\rsf F$ and $g\in \rsf 
	G_{1}\cap\rsf G_{2}$. Then
	$(h_{1}\meet h_{2})\join g= (h_{1}\join g)\meet (h_{2}\join g)= 
	\one\meet\one=\one$ and so 
	$(h_{1}\meet h_{2}\in (\rsf G_{1}\cap\rsf G_{2})\to\rsf F$. 
	
	Conversely, let $h\join g=\one$ for all $g\in\rsf G_{1}\cap\rsf G_{2}$. 
	As $\rsf G_{i}$ are both $\rsf F$-Boolean there exists $h_{i}\in\rsf 
	G_{i}\to\rsf F$ and $g_{i}\in\rsf G_{i}$ with
	$h= h_{1}\meet g_{1}= h_{2}\meet g_{2}$. Then
	\begin{align*}
		h_{1}\meet h_{2}\meet(g_{1}\join g_{2})&=
		(h_{1}\meet h_{2}\meet g_{1})\join(h_{1}\meet h_{2}\meet g_{2})\\
		&=(h_{2}\meet h)\join(h_{1\meet h})\\
		&=h\meet h=h. 
	\end{align*}
	As $h_{1}\meet h_{2}\in (\rsf G_{1}\to\rsf F)\join(\rsf G_{2}\to\rsf F)$
	and $g_{1}\join g_{2}\in \rsf G_{1}\cap\rsf G_{2}$
	we then have 
	$h= [h\join (h_{1}\meet h_{2})]\meet (h\join g_{1}\join g_{2})=
	h\join (h_{1}\meet h_{2})$ and so $h=h_{1}\meet h_{2}$ is in 
	$(\rsf G_{1}\to\rsf F)\join(\rsf G_{2}\to\rsf F)$.	
\end{proof}

\begin{cor}\label{cor:interStrBool}
	Let $\rsf G_{1}$ and $\rsf G_{2}$ be $\rsf F$-Boolean. Then so is
	$\rsf G_{1}\cap\rsf	G_{2}$. 
\end{cor}
\begin{proof}
	Let $f\in\rsf F$. 
	As $\rsf G_{i}$ are both $\rsf F$-Boolean there exists $h_{i}\in\rsf 
	G_{i}\to\rsf F$ and $g_{i}\in\rsf G_{i}$ with
	$f= h_{1}\meet g_{1}= h_{2}\meet g_{2}$. Then as above
	$f= h_{1}\meet h_{2}\meet(g_{1}\join g_{2})$ and
	$g_{1}\join g_{2}\in\rsf G_{1}\cap\rsf G_{2}$ and 
	$h_{1}\meet h_{2}\in (\rsf G_{1}\to\rsf F)\join(\rsf G_{2}\to\rsf F)= (\rsf G_{1}\cap\rsf 
	G_{2})\to\rsf F$.
\end{proof}

\begin{cor}\label{cor:joinStrBool}
	Let $\rsf G_{1}$ and $\rsf G_{2}$ be $\rsf F$-Boolean. Then so is
	$\rsf G_{1}\join\rsf G_{2}$. 
\end{cor}
\begin{proof}
	Since we have $(\rsf G\to\rsf F)\to\rsf F=\rsf G$ for $\rsf 
	F$-Booleans we know that
	$\rsf G_{i}\to\rsf F$ are also $\rsf F$-Boolean and so
	\begin{align*}
		\rsf G_{1}\join\rsf G_{2} & = ((\rsf G_{1}\to\rsf F)\to\rsf F)\join 
		((\rsf G_{2}\to\rsf F)\to\rsf F)\\
		 & = ((\rsf G_{1}\to\rsf F)\cap
		 		(\rsf G_{2}\to\rsf F))\to\rsf F \\
				\intertext{Therefore}
		(\rsf G_{1}\join\rsf G_{2})\to\rsf F & =
		(((\rsf G_{1}\to\rsf F)\cap
		 		(\rsf G_{2}\to\rsf F))\to\rsf F)\to\rsf F\\
		 & =(\rsf G_{1}\to\rsf F)\cap
		 		(\rsf G_{2}\to\rsf F). 
	\end{align*}
	Thus we have 
	$$
	(\rsf G_{1}\join\rsf G_{2})\join((\rsf G_{1}\join\rsf G_{2})\to\rsf F)=
	(\rsf G_{1}\join\rsf G_{2})\join((\rsf G_{1}\to\rsf F)\cap
		 		(\rsf G_{2}\to\rsf F)). 
	$$
	Let $f\in\rsf F$ and $g_{i}\in\rsf G_{i}$, $h_{i}\in\rsf G_{i}\to\rsf 
	F$ be such that $f=g_{i}\meet h_{i}$. Then
	$g_{1}\meet g_{2}\in\rsf G_{1}\join\rsf G_{2}$, 
	$h_{1}\join h_{2}\in (\rsf G_{1}\to\rsf F)\cap
		 		(\rsf G_{2}\to\rsf F)$ and 
	\begin{align*}
		g_{1}\meet g_{2}\meet (h_{1}\join h_{2}) & =
		(g_{1}\meet g_{2}\meet h_{1})\join(g_{1}\meet g_{2}\meet h_{2})\\
		&=(g_{2}\meet f)\join(g_{1}\meet f)\\
		&=f\meet f&&\text{ as }f\le g_{i}\\
		&=f. 
	\end{align*}
\end{proof}

Thus we have 
\begin{thm}\label{thm:BooleanAlgebra}
	Let $\rsf F$ be any filter. Then
	$\Set{\rsf G | \rsf G\text{ is }\rsf F\text{-Boolean}}$ ordered by reverse 
	inclusion is a Boolean algebra with $\meet=\join$, $\join=\cap$, 
	$1=\Set{\one}$, $0=\rsf F$ and $\comp{\rsf G}=\rsf G\to\rsf F$. 
\end{thm}
\begin{proof}
	This is immediate from lemma \ref{lem:distrib} and preceding remarks, 
	and from \lemref{lem:BooleanANotherWay}.  
\end{proof}

We need a stronger closure property for Boolean filters under 
intersection. 
\begin{lem}\label{lem:closIntersect}
	Let $\rsf F\sim\rsf H$, $\rsf G$ be $\rsf F$-Boolean and $\rsf K$ 
	be $\rsf H$-Boolean. Then $\rsf G\cap\rsf K$ is $\rsf F\cap\rsf 
	H$-Boolean. 
\end{lem}
\begin{proof}
	Let $p\in\rsf F\cap\rsf H$ be arbitrary. Choose
	$g\in\rsf G$, $g'\in\rsf G\to\rsf F$ with $g\meet g'=p$ and choose
	$k\in\rsf K$, $k'\in\rsf K\to\rsf H$ with $k\meet k'=p$. 
	
	Then $g'$ and $k'$ are both above $p$ so $g'\meet k'$ exists and is 
	is $\rsf F\cap\rsf H$. Also
	$(g\join k)\meet(g'\meet k')=p$. $g\join k\in\rsf G\cap\rsf K$ so we 
	need to show that $g'\meet k'$ is in $(\rsf G\cap\rsf K)\to(\rsf 
	F\cap\rsf H)$. Let 
	$q\in\rsf G\cap\rsf K$. Then $q\join g'=\one=q\join k'$ so that
	$q\join (g'\meet k')= (q\join g')\meet(q\join k')= \one$. 
\end{proof}

\begin{cor}\label{cor:closIntersect}
	Let $\rsf G$ and $\rsf K$ be Boolean. Then so is $\rsf G\cap\rsf K$. 
\end{cor}
\begin{proof}
	Let $\rsf F\sim\rsf H$ be two g-filters such that 
	$\rsf G\subseteq\rsf F$ and $\rsf K\subseteq\rsf H$. Then
	the lemma gives $\rsf G\cap \rsf K$ to be $\rsf F\cap\rsf H$-Boolean. 
	Theorem \ref{thm:lots} tells us that $\rsf F\cap\rsf H$ is $\rsf F$-Boolean.
	And from \lemref{lem:middle} we have $\rsf G\cap\rsf K$ to be $\rsf 
	F$-Boolean. 
\end{proof}

The last closure property we need is with respect to $\Delta$. 
\begin{lem}\label{lem:deltaGivesStrA}
	Let $\rsf G\subseteq\rsf H\subseteq\rsf F$ be $\rsf F$-Boolean 
	filters. Then
	$$
		\Delta(\rsf G, \rsf H)\to\Delta(\rsf G, \rsf F)=\Delta(\one, \rsf 
		H\to\rsf F). 
	$$
\end{lem}
\begin{proof}
	As $\rsf G\subseteq\rsf H\subseteq\rsf F$ in a Boolean algebra we have
	$$
	(\rsf G\to\rsf H)\to(\rsf G\to\rsf F)=\rsf H\to\rsf F. 
	$$
	Also we have 
	\begin{align*}
		\Delta(\rsf G, \rsf H) & =\rsf G\join\Delta(\one, \rsf G\to\rsf H)  \\
		\Delta(\rsf G, \rsf F) & =\rsf G\join\Delta(\one, \rsf G\to\rsf F). 
	\end{align*}
	Let $x\in\Delta(\rsf G, \rsf H)$ and $g\in \rsf G$, $h\in\rsf G\to\rsf H $
	with $x=g\meet\Delta(\one, h)$. 
	Let $y\in\Delta(\rsf G, \rsf F)$ and $g'\in \rsf G$, $f\in\rsf G\to\rsf 
	F $ with $y=g'\meet\Delta(\one, f)$ and suppose that $x\join y=\one$ 
	for all such $x$. 
	Then 
	\begin{align*}
		y\join x & = (g'\meet\Delta(\one, f))\join
		(g\meet\Delta(\one, h))\\
		 & =(g'\join g)\meet(g'\join\Delta(\one, h))\meet
		 (\Delta(\one, f)\join g)\meet\Delta(\one, f\join h)\\
		 & =(g'\join g)\meet\Delta(\one, f\join h) && \text{ since }g\meet 
		 f\text{ and }g'\meet h\text{ exists.}
	\end{align*}
	Thus $g'\join g=\one$ and $f\join h=\one$ for all $g\in\rsf G$ and 
	all $h\in\rsf G\to\rsf H$. Choosing $g=g'$ implies $g'=\one$
	and so
	$f\in (\rsf G\to\rsf H)\to(\rsf G\to\rsf F)= \rsf H\to\rsf F$. 
	Hence $y=\Delta(\one, f)\in\Delta(\one, \rsf H\to\rsf F)$. 
	
	Conversely if $f\in\rsf H\to\rsf F$ then $g\join \Delta(\one, 
	f)=\one$ for all $g\in\rsf G$. And 
	$f\in (\rsf G\to\rsf H)\to(\rsf G\to\rsf F)$ implies $h\join f=\one$ 
	for all
	$h\in\rsf G\to\rsf H$. Hence $(g\meet\Delta(\one, h))\join\Delta(\one, 
	f)=\one$ and so 
	$\Delta(\one, f)$ is in $\Delta(\rsf G, \rsf H)\to\Delta(\rsf G, \rsf F)$. 
\end{proof}

\begin{lem}\label{lem:deltaGivesStr}
	Let $\rsf G\subseteq\rsf H\subseteq\rsf F$ be $\rsf F$-Boolean 
	filters. Then
	$\Delta(\rsf G, \rsf H)$ is $\Delta(\rsf G, \rsf F)$-Boolean. 
\end{lem}
\begin{proof}
	Since 
	\begin{align*}
		\Delta(\rsf G, \rsf H)\join(\Delta(\rsf G, \rsf H)\to\Delta(\rsf G, \rsf F)) & =
		\rsf G\join\Delta(\one, \rsf G\to\rsf H)\join
		\Delta(\one, \rsf H\to\rsf F)\\
		 & =
		\rsf G\join\Delta(\one, \rsf G\to\rsf H)\join
		\Delta(\one, (\rsf G\to\rsf H)\to(\rsf G\to\rsf F))  \\
		 & = \rsf G\join\Delta((\one, \rsf G\to\rsf H)\join
		((\rsf G\to\rsf H)\to(\rsf G\to\rsf F)))  \\
		 & =\rsf G\join\Delta(\one, \rsf G\to\rsf F)\\
		 &= \Delta(\rsf G, \rsf F). 
	\end{align*}
\end{proof}

From this lemma  we can derive another property of $\Delta$. 
\begin{lem}\label{lem:implDDD}
	Let $\rsf G\subseteq\rsf H\subseteq\rsf F$ be $\rsf F$-Boolean 
	filters. Then
	$$
		\rsf G\to\Delta(\rsf H, \rsf F)=(\rsf G\to\rsf H)\join\Delta(\one, 
		\rsf H\to\rsf F). 
	$$
\end{lem}
\begin{proof}
	The RHS is clearly a subset of $\Delta(\rsf H, \rsf F)$. 
	Let $g\in\rsf G$. If $h\in\rsf G\to\rsf H$ then $h\join g=\one$. 
	If $k\in \Delta(\one, \rsf H\to\rsf F)$ then $g\meet\Delta(\one, k)$ 
	exists so that $g\join k=\one$. Thus the RHS is a subset of the LHS. 
	
	Conversely suppose that $h=h_{1}\meet h_{2}$ is in $\rsf H\join\Delta(\one, 
	\rsf H\to\rsf F)= \Delta(\rsf H, \rsf F)$ and $g\join h=\one$
	for all $g\in\rsf G$. Then $g\join h_{1}=\one$ for all $g\in\rsf G$ 
	and so $h_{1}\in\rsf G\to\rsf H$. Thus the LHS is a subset of the RHS. 
\end{proof}

\begin{cor}\label{cor:iteratedDelta}
	Let $\rsf G\subseteq\rsf H\subseteq\rsf F$ be $\rsf F$-Boolean 
	filters. Then
	$$
		\Delta(\rsf G, \Delta(\rsf H, \rsf F))=\Delta(\Delta(\rsf G, \rsf H), 
		\Delta(\rsf G, \rsf F)). 
	$$
\end{cor}
\begin{proof}
	\begin{align*}
		\Delta(\rsf G, \Delta(\rsf H, \rsf F)) & =
		\rsf G\join\Delta(\one, \rsf G\to\Delta(\rsf H, \rsf F))\\
		 & =\rsf G\join\Delta(\one, (\rsf G\to\rsf H)\join\Delta(\one, 
		\rsf H\to\rsf F))  \\
		 & =\rsf G\join \Delta(\one, \rsf G\to\rsf H)\join(\rsf H\to\rsf F)  \\
		 & = \Delta(\rsf G, \rsf H)\join\Delta(\one, \Delta(\rsf G, \rsf H)\to
		 \Delta(\rsf G, \rsf F)) \\
		 & = \Delta(\Delta(\rsf G, \rsf H), \Delta(\rsf G, \rsf F)). 
	\end{align*}
\end{proof}

\subsection{Weakly Boolean Elements}
The theory of weakly Boolean filters just follows the usual theory of 
Boolean elements in the Heyting algebra of ideals of a Boolean algebra. 

However we also want to consider how the $\rsf F$-Boolean filters relate to 
the weakly $\rsf F$-Boolean filters. 
By \lemref{lem:BooleanANotherWay} every $\rsf F$-Boolean filter is 
weakly $\rsf F$-Boolean. There are many other weakly Boolean filters 
as the next lemma shows:
\begin{lem}\label{lem:tripleImpl}
	$((\rsf G\to\rsf F)\to\rsf F)\to\rsf F=\rsf G\to\rsf F$. 
\end{lem}
\begin{proof}
	We know that $\rsf G\to\rsf F\subseteq((\rsf G\to\rsf F)\to\rsf F)\to\rsf F$. 
	Also $\rsf G\subseteq (\rsf G\to\rsf F)\to\rsf F$ implies 
	$((\rsf G\to\rsf F)\to\rsf F)\to\rsf F\subseteq\rsf G\to\rsf F$.
\end{proof}

Thus for any filter $\rsf G$ we see that $\rsf G\to\rsf F$ is weakly 
Boolean. This gives a natural closure operator
$\rsf G\mapsto (\rsf G\to\rsf F)\to\rsf F=\text{cl}(\rsf G)$ from filters to weakly 
Boolean filters. 

To naturally define operations on weakly Boolean filters we have 
$\rsf G_{1}{\pmb\vee}\rsf G_{2}=\text{cl}(\rsf G_{1}\join\rsf G_{2})$. 

For meets we use the usual meets as we have:
\begin{lem}\label{lem:meetWkBool}
	Let $\rsf G_{1}$ and $\rsf G_{2}$ be two weakly $\rsf F$-Boolean 
	filters. Then $\rsf G_{1}\cap\rsf G_{2}$ is weakly $\rsf F$-Boolean. 
\end{lem}
\begin{proof}
	$\rsf G_{1}\cap\rsf G_{2}\subseteq\text{cl}(\rsf G_{1}\cap\rsf 
	G_{2})$ is always true. 
	
	$\rsf G_{1}\cap\rsf G_{2}\subseteq\rsf G_{i}$ and so
	$\rsf G_{i}\to\rsf F\subseteq(\rsf G_{1}\cap\rsf G_{2})\to\rsf F$ 
	and hence 
	$\text{cl}(\rsf G_{1}\cap\rsf G_{2})\subseteq\text{cl}(\rsf 
	G_{i})=\rsf G_{i}$. 
	Hence $\text{cl}(\rsf G_{1}\cap\rsf G_{2})\subseteq\rsf 
	G_{1}\cap\rsf G_{2}$. 
\end{proof}

\subsection{An MR-algebra}
The above lemmas show us that the set of weakly $\rsf F$-Boolean 
filters with operations $\cap$ and $\mathbf\vee$ forms a Boolean 
algebra into which the Boolean algebra of $\rsf F$-Boolean filters 
embeds. The larger algebra is (as usual) complete. 
These lemmas also show us that there is a natural cubic algebra here. 
\begin{thm}\label{thm:MRalgFilter}
	Let $\mathcal L_{sB}$ be the set of all Boolean filters ordered by 
	reverse inclusion. 
	Then 
	\begin{enumerate}[(a)]
		\item $\mathcal L_{sB}$ contains $\Set{\one}$ and is closed
		 under the operations $\join$ and $\Delta$. 
		
		\item $\brk<\mathcal L_{sB}, \Set{\one}, \join, \Delta>$ is an atomic MR-algebra. 
	
		\item The mapping $e\colon\mathcal L\to \mathcal L_{sB}$ given by 
		$g\mapsto[g, \one]$ is a full embedding. 
	
		\item The atoms of $\mathcal L_{sB}$ are exactly the g-filters of 
		$\mathcal L$. 
	\end{enumerate}
\end{thm}
\begin{proof}
	\begin{enumerate}[(a)]
		\item It is easy to see that $\one\to\rsf F=\rsf F$ for all filters 
		$\rsf F$. Corollary \ref{cor:closIntersect} and \lemref{lem:deltaGivesStr}
		give the closure under join and Delta respectively. 
		
		\item 
		We will proceed sequentially through the axioms. 
		\begin{enumerate}[i.]
        	\item  if $x\le y$ then $\Delta(y, x)\join x = y$ -- this is 
        	\lemref{lem:interDelta}. 
        
        	\item  if $x\le y\le z$ then $\Delta(z, \Delta(y, x))=\Delta(\Delta(z, 
        		y), \Delta(z, x))$ -- this is \corref{cor:iteratedDelta}. 
        
        	\item  if $x\le y$ then $\Delta(y, \Delta(y, x))=x$ -- this 
        	is \corref{cor:doubleDelta} and the definition of $\rsf 
        	F$-Boolean. 
        
        	\item  if $x\le y\le z$ then $\Delta(z, x)\le \Delta(z, 
        	y)$ -- this is \lemref{lem:inclDelta}. 
        
        	\item[] Let $xy=\Delta(1, \Delta(x\join y, y))\join y$ for any $x$, $y$ 
        	in $\mathcal L$. 
			
			First we note that if $\rsf G\subseteq\rsf F$ then
			\begin{align*}
				\Delta(\one, \Delta(\rsf G, \rsf F))\cap\rsf F & =
				\Delta(\one, \rsf G\join\Delta(\one, \rsf G\to\rsf F))\cap\rsf F\\
				 & = (\Delta(\one, \rsf G)\join(\rsf G\to\rsf F))\cap\rsf F . 
			\end{align*}
			If $g\in\rsf G$ and $h\in\rsf F$ is such that $\Delta(\one, g)\meet h\in\rsf F$ 
			then $g=\Delta(\one, g)$ (since $g\simeq\Delta(\one, g)$ and
			$g\meet \Delta(\one, g)$ exists). Thus 
			$(\Delta(\one, \rsf G)\join(\rsf G\to\rsf F))\cap\rsf F= 
			\rsf G\to\rsf F$. 
        
        	\item  $(xy)y=x\join y$ and 
        	\item  $x(yz)=y(xz)$. 
			These last two properties hold as $\mathcal L_{sB}$ is locally 
			Boolean and hence an implication algebra. 
    	\end{enumerate}
		
		To see that $\mathcal L_{sB}$ is an MR-algebra it suffices to note 
		that if $\rsf G_{1}$ and $\rsf G_{2}$ are in $\mathcal L_{sB}$ and we have 
		g-filters $\rsf F_{1}, \rsf F_{2}$ with $\rsf G_{i}\subseteq\rsf 
		F_{i}$ then $\Delta(\rsf F_{1}\cap\rsf F_{2}, \rsf F_{2})= \rsf 
		F_{1}\supseteq\rsf G_{1}$ so that $\rsf F_{2}\preccurlyeq\rsf G_{1}$. 
		It is clear that $\rsf F_{2}\preccurlyeq\rsf G_{2}$.
	
		\item It is clear that this mapping preserves 
		order and join. Preservation of $\Delta$ is \corref{cor:DeltaDoublePrinc}. 
		
		It is full because $[g, \one]\subseteq\rsf G$ whenever $g\in\rsf 
		G$. 
	
		\item This is \thmref{thm:lots}. 
	\end{enumerate}
\end{proof}

\section{The Algebra of g-filters}
Earlier we looked at the group of g-filters with an operation that 
comes from the $\beta$ functions. In all atomic MR-algebras there 
is a natural isomorphism between 
$[v, \one]$ and the set of atoms (or vertices) given by
$x\mapsto\Delta(x, v)$ with inverse $w\mapsto w\join v$. 

Thus the theorem above gives us another 
way to make the set of g-filters into an algebra -- by defining
\begin{align*}
	0 & =\rsf F  \\
	1 & =\Delta(\one, \rsf F)  \\
	\rsf H+\rsf G & =\Delta((\rsf H\cap\rsf F)+(\rsf G\cap\rsf F), \rsf F)  \\
	\rsf H\cdot\rsf G & =\Delta((\rsf H\cap\rsf F)\join(\rsf G\cap\rsf F), \rsf F)
\end{align*}
This makes an algebra isomorphic to $[\rsf F, \one]$. 

We want to show that this is an extension of the group the 
operations $+$ and $*$ are the same. 

\begin{thm}\label{thm:sameOps}
	Let $\rsf G$ and $\rsf H$ be two g-filters. Then 
	$$
		\rsf G*\rsf H=\rsf G+\rsf H. 
	$$
\end{thm}
\begin{proof}
	We know that 
	$\Delta((\rsf H\cap\rsf F)+(\rsf G\cap\rsf F), \rsf F)= \Delta(\one, 
	\rsf K\to\rsf F)\join\rsf K$ where 
	$\rsf K$ is the filter $(\rsf H\cap\rsf F)+(\rsf G\cap\rsf F)$. 
	
	Also
	$$
		\rsf G*\rsf H=\Set{\Delta(\beta_{\rsf G}(\alpha_{\rsf F}(h)), \beta_{\rsf 
	G}(h)) | h\in\rsf H}.
	$$
	
	As both sides are g-filters, it suffices to show that 
	\begin{align}
		\rsf K & \subseteq \rsf G*\rsf H
		\label{eq:firstIncl}  \\
		\rsf K\to\rsf F & \subseteq\Delta(\one, \rsf G*\rsf H). 
		\label{eq:secIncl}
	\end{align}
	
	First we elaborate on $\rsf K$. 
	\begin{align*}
		\rsf K & = (\rsf H\cap\rsf F)+(\rsf G\cap\rsf F) \\
		 & =\Bigl[\bigl[(\rsf G\cap\rsf F)\to\rsf F\bigl]\join(\rsf H\cap\rsf 
		 F)\Bigr]\cap  \\
		 &\qquad\qquad
		 \Bigl[\bigl[(\rsf H\cap\rsf F)\to\rsf F\bigl]\join(\rsf G\cap\rsf 
		 F)\Bigr]\\
		[(\rsf G\cap\rsf F)\to\rsf F]\join(\rsf H\cap\rsf F) & =\Set{g\in\rsf F | 
		g=g_{1}\meet g_{2}\text{, }g_{2}\in\rsf H\cap\rsf F\text{ and 
		}\forall k\in\rsf G\cap\rsf F\ g_{1}\join k=\one}\\
		\intertext{Thus }
		g\in\rsf K & \leftrightarrow 
		g=g_{1}\meet g_{2}\text{, }g_{2}\in\rsf H\cap\rsf F\text{ and 
		}\forall k\in\rsf G\cap\rsf F\ g_{1}\join k=\one\\
		\text{ and }\qquad&\hphantom{\leftrightarrow}
		g=h_{1}\meet h_{2}\text{, }h_{2}\in\rsf G\cap\rsf F\text{ and 
		}\forall k\in\rsf H\cap\rsf F\ h_{1}\join k=\one. 
	\end{align*}
	
	\begin{description}
		\item[(1)] Let $k\in\rsf K$. We want to show that $k\in \rsf G*\rsf H$. 
		
		Then we have 
		\begin{align*}
			k\join\beta_{\rsf H}(k) & =(h_{1}\join\beta_{\rsf H}(k))\meet
			(h_{2}\join\beta_{\rsf H}(k))\\
			 & =h_{2}\join\beta_{\rsf H}(k). \\
			 k\join\beta_{\rsf G}(k) & =(g_{1}\join\beta_{\rsf G}(k))\meet
			(g_{2}\join\beta_{\rsf G}(k))\\
			 & =g_{2}\join\beta_{\rsf G}(k).
		\end{align*}
		Also $h_{2}\in\rsf G$, $h_{2}\geq k$ so that $\beta_{\rsf G}(k)\preccurlyeq 
		h_{2}$ and $\beta_{\rsf G}(k)\meet h_{2}$ exists. Thus 
		$\beta_{\rsf G}(k)\le h_{2}$. Likewise $\beta_{\rsf H}(k)\le g_{2}$. 
		
		We also have 
		\begin{align*}
			k&= (h_{1}\meet h_{2})\join(g_{1}\meet g_{2})\\
			&= (h_{1}\join g_{1})\meet(h_{2}\meet g_{2})\\
			\intertext{ so that }
			\alpha_{\rsf H}(k)=k\join\beta_{\rsf H}(k) & = (h_{1}\join g_{1}\join\beta_{\rsf H}(k))
			\meet(h_{2}\meet g_{2}\join\beta_{\rsf H}(k)) \\
			 & = h_{2}\join g_{2}\join \beta_{\rsf H}(k)  \\
			 & = h_{2}\join g_{2}. \\
			 \intertext{Likewise we have }
			 \alpha_{\rsf G}(k)=k\join\beta_{\rsf G}(k) & = h_{2}\join g_{2}.
		\end{align*}
		Thus $\alpha_{\rsf G}(k)=\alpha_{\rsf H}(k)$. 
		Let $h=\beta_{\rsf H}(k)$. Then 
		$\beta_{\rsf G}(h)=\beta_{\rsf G}(k)$ and 
		$\alpha_{\rsf H}(k)= k\join\beta_{\rsf H}(k)= \beta_{\rsf F}(h)\join 
		h= \alpha_{\rsf F}(h)$. Hence we have 
		\begin{align*}
			\Delta(\beta_{\rsf G}\alpha_{\rsf F}(h), \beta_{\rsf G}(h)) & =
			\Delta(\beta_{\rsf G}\alpha_{\rsf H}(k), \beta_{\rsf G}(k))\\
			 & =\Delta(\beta_{\rsf G}\alpha_{\rsf G}(k), \beta_{\rsf G}(k))  \\
			 & =\Delta(\alpha_{\rsf G}(k), \beta_{\rsf G}(k))  \\
			 & =k. 
		\end{align*}
	
		\item[2] 
		\begin{enumerate}[(a)]
		\item	Now we show that $\rsf G*\rsf H\cap\rsf F\subseteq\rsf K$. The 
		reverse inclusion follows from the last part. 
		
		Let $$
			r=
			\begin{cases}
				\Delta(\beta_{\rsf G}\alpha_{\rsf F}(h), \beta_{\rsf G}(h)) & \text{ 
				for some }h\in\rsf H  \\
				 \Delta(\beta_{\rsf H}\alpha_{\rsf F}(g), \beta_{\rsf H}(g)) & \text{ 
				for some }g\in\rsf G 
			\end{cases}
		$$
		be in $\rsf G*\rsf H\cap\rsf F$. We recall that for any g-filter 
		$\rsf Q$ we have 
		\begin{align*}
			\Delta(\one, x)\join\beta_{\rsf Q}(x) & =\alpha_{\rsf Q}(x)\to\beta_{\rsf Q}(x)  \\
			x \join\beta_{\rsf Q}(x) & =\alpha_{\rsf Q}(x). 
		\end{align*}
		From this and the above equations we have 
		\begin{align*}
			\alpha_{\rsf G}(r)\to\beta_{\rsf G}(r) &= \Delta(\one, 
			r)\join\beta_{\rsf G}(r)  \\
			r \join\beta_{\rsf G}(r) & =\alpha_{\rsf G}(r) \\
			\Delta(\one, \alpha_{\rsf G}(r)\to\beta_{\rsf G}(r)) &= r \join\Delta
			(\one, \beta_{\rsf G}(r)) \\
			r & = \alpha_{\rsf G}(r)\meet 
			\Delta(\one, \alpha_{\rsf G}(r)\to\beta_{\rsf G}(r)) \\
			\intertext{These are both in $\rsf F$ as they are both greater than $r$. 
			Clearly $\alpha_{\rsf G}(r)\in\rsf G\cap\rsf F$. We want 
			to show that 
			$p=\Delta(\one, \alpha_{\rsf G}(r)\to\beta_{\rsf G}(r))$ is in 
			$(\rsf H\cap\rsf F)\to\rsf F$. We also have }
			\alpha_{\rsf G}(r) & =\beta_{\rsf G}(\alpha_{\rsf F}(h))  \\
			\beta_{\rsf G}(r) & =\beta_{\rsf G}(h)\\
			\beta_{\rsf G}(p) &= \alpha_{\rsf G}(r)\to\beta_{\rsf G}(r). 
		\end{align*}
		Now pick any $h'\in\rsf H\cap\rsf F$. Then we have 
		\begin{align*}
			\beta_{\rsf G}(p\join h') & = \beta_{\rsf G}(p)\join \beta_{\rsf 
			G}(h') && \text{from $\beta_{\rsf G}\colon\rsf F\to\rsf G$} \\
			 & =  (\alpha_{\rsf G}(r)\to\beta_{\rsf G}(r))\join\beta_{\rsf G}(h') \\
			 & =(\beta_{\rsf G}(\alpha_{\rsf F}(h))\to\beta_{\rsf 
			 G}(h))\join\beta_{\rsf G}(h') \\
			 & =\beta_{\rsf G}((\alpha_{\rsf F}(h)\to h)\join h') && \text{from 
			 $\beta_{\rsf G}\colon\rsf 
			 			 H\to\rsf G$}  \\
			 \alpha_{\rsf F}(h)\to h & =(h\join\beta_{\rsf F}(h))\to h  \\
			 & =\Delta(\one, \Delta(h\join\beta_{\rsf F}(h), h))\join h\\
			 &= \Delta(\one, \beta_{\rsf F}(h))\join h\\
			 \intertext{Thus we have }
			 (\alpha_{\rsf F}(h)\to h)\join h'&=h\join\Delta(\one, \beta_{\rsf 
			 F}(h))\join h'\\
			 &=\one &&\text{since $h'\in\rsf F$ and}\\
			 &&&\text{so the meet $\beta_{\rsf 
			 F}(h)\meet h'$ exists. }\\
			 \intertext{Hence }
			 p\join h' &= \one. 
		\end{align*}
		It follows that $r\in(\rsf G\cap\rsf F)\join((\rsf H\cap\rsf 
		F)\to\rsf F)$. Dually we have 
		$r\in(\rsf H\cap\rsf F)\join((\rsf G\cap\rsf 
		F)\to\rsf F)$ and so $r\in\rsf K$. 
		
		\item $\rsf K\to\rsf F\subseteq \Delta(\one, \rsf G*\rsf H)$. 
		We appeal to \lemref{lem:lots}. This gives 
		\begin{align*}
			\rsf K\to \rsf F & =(\rsf G*\rsf H\cap\rsf F)\to\rsf F  \\
			 & =\Delta(\one, \rsf G*\rsf H)\cap\rsf F  \\
			 & \subseteq\Delta(\one, \rsf G*\rsf H). 
		\end{align*}
		
		\item Thus we have $\Delta(\rsf K, \rsf F)= 
		\rsf K\join\Delta(\one, \rsf K\to\rsf F)\subseteq \rsf G*\rsf H$. 
		Since both sides are g-filters they must be equal. This can also be 
		seen from \corref{cor:lots}. 
	\end{enumerate}
	\end{description}
\end{proof}

Earlier we showed that the relation $\rsf G\sim\rsf H$ is a congruence 
on the group of g-filters. 
We can then look at the subgroup $\mathcal N=\Set{\rsf K | \rsf 
K\sim\rsf F}$. From the isomorphism with $[\rsf F, \one]$, the theorem above and from
\thmref{thm:unTwist} and \corref{cor:gFilter}, \lemref{lem:gFilter} 
we have 
\begin{align*}
	\rsf G\sim\rsf F & \iff \rsf G\cap\rsf F\text{ or }\Delta(\one, \rsf 
	G)\cap\rsf F\text{ is principal}  \\
	 & \iff \exists g\in\rsf F\ \rsf G\cap\rsf F=[g, \one]\text{ or }\Delta(\one, \rsf 
	G)\cap\rsf F=[g, \one]  \\
	& \iff \exists g\in\rsf F\ \rsf G=\Delta([g, \one], \rsf F)\text{ or }
	\Delta(\one, \rsf G)=\Delta([g, \one], \rsf F). 
\end{align*}
Thus we see that the corresponding subalgebra of $[\rsf F, \one]$ is
$\Set{[g, \one] | g\in\rsf F}\cup\Set{g\to\rsf F | g\in\rsf F}$. 

This is a subalgebra, but is not upwards or downward closed and so 
does not induce a Boolean congruence on $[\rsf F, \one]$. 

Thus we have shown that the algebra of filters of a filter MR-algebra 
has properties very like those of the algebra of filters of a Boolean 
algebra. We are able to recover a Boolean component of that poset 
that gives a natural MR-algebra into which our original structure 
embeds as an upwards-dense subalgebra. We expect that many of these 
ideas can be extended to larger families of cubic algebras,  
particularly those that are generated by implication algebras. 

However we also know that there are MR-algebras that cannot be 
represented as filter algebras and for these some new idea is clearly 
needed.

\end{document}